\documentclass[a4paper, 11pt,reqno]{amsart}
\usepackage{amsfonts}
\usepackage{amsmath,amsthm}
\usepackage{amssymb}
\usepackage{amscd}
\usepackage{enumerate}
\usepackage{hyperref}

\hypersetup{
    colorlinks=true,           linkcolor=red,              citecolor=blue,            filecolor=magenta,          urlcolor=cyan           }
\newtheorem{theorem}{Theorem}[section]

\newtheorem{definition}[theorem]{Definition}

\newtheorem{lemma}[theorem]{Lemma}

\newtheorem{proposition}[theorem]{Proposition}
\newtheorem{remark}[theorem]{Remark}

\def\nc{\newcommand}

\def\be{\beta}

\nc\pa{\partial}
\nc\CC{\mathbb{C}}
\nc\RR{\mathbb{R}}
\nc\QQ{\mathbb{Q}}
\nc\ZZ{\mathbb{Z}}
\nc\NN{\mathbb{N}}

\def\ba{\begin{align}}
\def\bad{\begin{aligned}}
\def\be{\begin{equation}}
\def\ea{\end{align}}
\def\ead{\end{aligned}}
\def\ee{\end{equation}}

\input{tcilatex}

\begin{document}
\title{Existence of global solutions and blow-up of solutions for coupled
systems of fractional diffusion equations}
\author{Ahmad Bashir}
\thanks{E-mail address: bashirahmad\_qau@yahoo.com}
\author{Ahmed Alsaedi}
\thanks{E-mail address: aalsaedi@hotmail.com }
\author{Mohamed Berbiche}
\thanks{E-mail address: mohamed.berbiche@univ-biskra.dz,
berbichemed@yahoo.fr }
\author{Mokhtar Kirane}
\thanks{E-mail address: mokhtar.kirane@univ-lr.fr, }
\date{}

\begin{abstract}
We study the Cauchy problem for a system of semi-linear coupled
fractional-diffusion equations with polynomial nonlinearities posed in $%
\mathbb{R}_{+}\times \mathbb{R}^{N}$. Under appropriate conditions on the
exponents and the orders of the fractional time derivatives, we present a
critical value of the dimension $N$, for which global solutions with small
data exist, otherwise solutions blow-up in finite time. Furthermore, the
large time behavior of global solutions is discussed.
\end{abstract}

\maketitle

\section{Introduction}

We consider the system 
\begin{equation}
\left\{ 
\begin{array}{c}
^{C}D_{0|t}^{\gamma _{1}}u-\Delta u=f(v),\quad t>0,\;x\in \mathbb{R}^{N}, \\ 
^{C}D_{0|t}^{\gamma _{2}}v-\Delta v=g(u),\quad t>0,\;x\in \mathbb{R}^{N},%
\end{array}%
\right.  \label{sys1}
\end{equation}%
subject to the initial conditions 
\begin{equation}
u(0,x)=u_{0}(x),\quad v(0,x)=v_{0}(x),\quad x\in \mathbb{R}^{N},
\label{initdat}
\end{equation}%
where $0<\gamma _{1},\gamma _{2}<1$, for $0<\alpha <1$, ${}^{C}D_{0|t}^{%
\alpha }u$ denotes the Caputo time fractional derivative defined, for an
absolutely continuous function $u$, by 
\begin{equation*}
\big(^{C}D_{0|t}^{\alpha }u\big)(t)=\frac{1}{\Gamma (1-\alpha )}%
\int_{0}^{t}(t-s)^{-\alpha }\partial _{t}u(s,\cdot )\,ds,\quad 0<\alpha <1,
\end{equation*}%
where $\Delta $ is the Laplace operator in $\mathbb{R}^{N}$. The functions $%
f(v)$ and $g(u)$ are the nonlinear source terms that will be determined
later, and $u_{0}$, $v_{0}$ are given functions.

Before we present our results and comment on them, let us dwell on existing
results concerning the limiting case $\gamma _{1}=\gamma _{2}=1$. Escobedo
and Herrero \cite{EscoHerrero} studied the existence of global solutions,
and blowing-up of solutions for the system 
\begin{eqnarray}
u_{t}-\Delta u &=&v^{p},\quad t>0,x\in \mathbb{R}^{N},\;v>0,  \notag \\
v_{t}-\Delta v &=&u^{q},\quad t>0,x\in \mathbb{R}^{N},\;u>0.
\label{EscoHerr}
\end{eqnarray}%
They have shown, in particular, that for 
\begin{equation*}
pq>1,\quad \frac{N}{2}\leq \frac{\max \{p,q\}+1}{pq-1},
\end{equation*}%
every nontrivial solution of \eqref{EscoHerr} blows-up in a finite time $%
T^{\ast }=T^{\ast }(\Vert u\Vert _{\infty },\Vert v\Vert _{\infty })$, in
the sense that 
\begin{equation*}
\limsup_{t\rightarrow T^{\ast }}\Vert u(t)\Vert _{\infty
}=\limsup_{t\rightarrow T^{\ast }}\Vert v(t)\Vert _{\infty }=+\infty .
\end{equation*}%
The work \cite{EscoHerrero} has been followed by works of Escobedo and
Herrero in a bounded domain, Escobedo and Levine \cite{EscobedoLevine} for
more general nonlinear forcing terms, Uda \cite{Uda}, Fila, Levine and Uda 
\cite{FilaUda} for differing diffusive coefficients, Lu \cite{Lu}, Lu and
Sleeman \cite{LuSleeman}, Mochizuki \cite{Mochizuki}, Mochizuki and Huang 
\cite{MochizukiHuang}, Takase and Sleeman \cite{TakaseSleeman,SleemanTakase}%
, Samarskii et al.\ \cite{Samarski}, and many other authors; see the review
papers \cite{Denglevine, Bandle, Pohozaev}.

Time-fractional differential equations/systems for global or blowing-up
solutions have been studied, for example, in \cite%
{Diaz_PV,EidelKoch,Gafiychuk,KLT,Magin,Metzler,Saada,Zacher,ZhangQuan}.

Kirane, Laskri and Tatar \cite{KLT} studied the more general system 
\begin{eqnarray}
^{C}D_{0|t}^{\gamma _{1}}u+(-\Delta )^{\beta /2}u &=&|v|^{p},\quad
t>0,\;x\in \mathbb{R}^{N},\;p>1,  \notag \\
^{C}D_{0|t}^{\gamma _{2}}v+(-\Delta )^{\gamma /2}v &=&|u|^{q},\quad
t>0,\;x\in \mathbb{R}^{N},\;q>1,  \label{sysKLT}
\end{eqnarray}%
(for the definition of $(-\Delta )^{\sigma /2}$, $1\leq \sigma \leq 2$ see 
\cite{KLT}) with nonnegative initial data, and proved the non-existence of
global solutions under the condition 
\begin{equation*}
pq>1,\quad N\leq \max \Big\{\frac{\frac{\gamma _{2}}{q}+\gamma _{1}-(1-\frac{%
1}{pq})}{\frac{\gamma _{2}}{\gamma qp^{\prime }}+\frac{\gamma _{1}}{\beta
q^{\prime }}},\frac{\frac{\gamma _{1}}{p}+\gamma _{2}-(1-\frac{1}{pq})}{%
\frac{\gamma _{1}}{\beta pq^{\prime }}+\frac{\gamma _{2}}{\gamma p^{\prime }}%
}\Big\rbrace,
\end{equation*}%
where $p+p^{\prime }=pp^{\prime }$ and $q+q^{\prime }=qq^{\prime }$.

Here, we consider problem \eqref{sys1}-\eqref{initdat} and will give
conditions relating the space dimension $N$ with parameters $\gamma _{1}$, $%
\gamma _{2}$, $p$, and $q$ for which the solution of \eqref{sys1}-%
\eqref{initdat} exists globally in time and satisfies $L^{\infty }$-decay
estimates. We also discuss blowing-up in finite time solutions with initial
data having positive average. Our study of the existence of global solutions
relies on the semigroup theory, while for the blow-up of solutions result,
we use the test function approach due to Zhang \cite{Qi_S_Zhang} and
developed by Mitidieri and Pohozaev \cite{Pohozaev}, and used by several
authors (see for example \cite{KLT,FinoKirane,Pang,ZhangQuan}). Our result
on blowing-up solutions improves the one obtained in \cite{KLT}. We should
mention that to the best of our knowledge there are no global existence and
large time behavior results for the time-fractional diffusion system with
two different fractional powers. The paper of Zhang et al. \cite{ZhangQuanLi}
does not treat the case of different time fractional operators. Also in \cite%
{ZhangQuanLi}, the authors do not obtain the decay rate of the solution in
the space $L^{\infty }(\mathbb{R}^{N})$. \newline
The rest of this paper is organized as follows. In section 2, we present
some preliminary lemmas. In section 3, we present the main results of this
paper. Finally, section 4 and section 5 are devoted to the proofs of small
data global existence and blow-up in finite time of the solutions of problem %
\eqref{sys1}-\eqref{initdat}.

Throughout the paper, $C$ will denote a positive constant. The space $L^{p}(%
\mathbb{R}^{N})$ $(1\leq p<\infty )$ will be equipped with the usual norm $%
\Vert u\Vert _{L^{p}(\mathbb{R}^{N})}^{p}=\displaystyle\int_{\mathbb{R}%
^{N}}|u(x)|^{p}dx$. The space $C_{0}(\mathbb{R}^{N})$ denotes the set of all
continuous functions decaying to zero at infinity, equipped with Chebychev's
norm $\Vert u\Vert _{\infty }$.

\section{Preliminaries}

The left-sided and right-sided Riemann-Liouville integrals (see \cite{SKM}),
for $\Psi \in L^{1}(0,T)$, $0<\alpha <1$, are defined as follows%
\begin{equation*}
(I_{0\mid t}^{\alpha }\Psi )(t)=\frac{1}{\Gamma (\alpha )}\int_{0}^{t}\frac{%
\Psi (\sigma )}{(t-\sigma )^{\alpha -1}}\,d\sigma , \\
(I_{t\mid T}^{\alpha }\Psi )(t)=\frac{1}{\Gamma (\alpha )}\int_{t}^{T}\frac{%
\Psi (\sigma )}{(\sigma -t)^{\alpha -1}}\,d\sigma ,
\end{equation*}%
respectively, $\Gamma $ stands for the Euler gamma function.

The left-handed and right-handed Riemann-Liouville derivatives (see \cite%
{SKM}), for $\Psi \in AC^{1}(\left[ 0,T\right] )$, $0<\alpha <1$, are
defined as follows:

\begin{equation*}
( D_{0|t}^{\alpha }\Psi )(t)= ( \frac{d}{dt}\circ I_{0\mid
t}^{1-\alpha}\Psi) (t), \\
( D_{t|T}^{\alpha }\Psi )(t) =- ( \frac{d}{dt}\circ I_{t\mid
T}^{1-\alpha}\Psi) (t),
\end{equation*}
respectively.

The Caputo fractional derivative for a function $\Psi \in AC^{1}([0,T])$ is
defined by 
\begin{equation*}
(^{C}D_{0|t}^{\alpha }\Psi )(t)=\frac{1}{\Gamma (1-\alpha )}\int_{0}^{t}%
\frac{\Psi ^{\prime }(\sigma )}{(t-\sigma )^{\alpha }}d\sigma , \\
(^{C}D_{t|T}^{\alpha }\Psi )(t)=-\frac{1}{\Gamma (1-\alpha )}\int_{t}^{T}%
\frac{\Psi ^{\prime }(\sigma )}{(\sigma -t)^{\alpha }}d\sigma .
\end{equation*}%
For $0<\alpha <1$ and $\Psi \in AC^{1}(\left[ 0,T\right] )$, we have 
\begin{equation*}
\big(D_{0|t}^{\alpha }\Psi \big)(t)=\frac{1}{\Gamma (1-\alpha )}\big[\frac{%
\Psi (0)}{t^{\alpha }}+\int_{0}^{t}\frac{\Psi ^{\prime }(\sigma )}{(t-\sigma
)^{\alpha }}d\sigma \big],
\end{equation*}%
and 
\begin{equation}
\big(D_{t|T}^{\alpha }\Psi \big)(t)=\frac{1}{\Gamma (1-\alpha )}\Big[\frac{%
\Psi (T)}{(T-t)^{\alpha }}-\int_{t}^{T}\frac{\Psi ^{\prime }(\sigma )}{%
(\sigma -t)^{\alpha }}d\sigma \Big].  \label{eq:3}
\end{equation}%
The Caputo derivative is related to the Riemann-Liouville derivative by 
\begin{equation*}
^{C}D_{0|t}^{\alpha }\Psi (t)=(D_{0|t}^{\alpha })(\Psi (t)-\Psi (0)),\quad 
\text{for }\Psi \in AC^{1}([0,T]).
\end{equation*}%
Let $0<\alpha <1$, $f\in AC^{1}([0,T])$ and $g\in AC^{1}([0,T])$. Then 
\begin{equation*}
\int_{0}^{T}f(t)(D_{0|t}^{\alpha
}g)(t)\,dt=\int_{0}^{T}g(t)(^{C}D_{t|T}^{\alpha }f)(t)\,dt+f(T)(I_{0\mid
T}^{1-\alpha }g)(T).
\end{equation*}%
If $f(T)=0$, then 
\begin{equation*}
\int_{0}^{T}f(t)(D_{0|t}^{\alpha
}g)(t)\,dt=\int_{0}^{T}g(t)(^{C}D_{t|T}^{\alpha }f)(t)\,dt.
\end{equation*}%
For later use, let 
\begin{equation*}
\varphi (t)=\Big(1-\frac{t}{T}\Big)_{+}^{l},\quad \text{for }t\geq 0,\quad
l\geq 2.
\end{equation*}%
By a direct calculation, we get 
\begin{equation*}
{}^{C}D_{t|T}^{\alpha }\varphi (t)=\frac{\Gamma (l+1)}{\Gamma (l+1-\alpha )}%
T^{-\alpha }\big(1-\frac{t}{T}\big)_{+}^{l-\alpha },\quad t\geq 0.
\end{equation*}%
Now, we present some properties of two special functions. The two parameter
Mittag-Leffler function \cite{SKM} is defined for $z\in \mathbb{C}$ as 
\begin{equation*}
E_{\alpha ,\beta }(z)=\sum_{k=0}^{\infty }\frac{z^{k}}{\Gamma (\alpha
k+\beta )},\quad \alpha ,\beta \in \mathbb{C},\;\Re (\alpha )>0.
\end{equation*}%
It satisfies 
\begin{equation*}
I_{0|t}^{1-\alpha }(t^{\alpha -1}E_{\alpha ,\alpha }(\lambda t^{\alpha
}))=E_{\alpha ,1}(\lambda t^{\alpha })\quad \text{for }\lambda \in \mathbb{C}%
,0<\alpha <1.
\end{equation*}%
The Wright type function 
\begin{align*}
\phi _{\alpha }(z)& =\sum_{k=0}^{\infty }\frac{(-z)^{k}}{k!\Gamma (-\alpha
k+1-\alpha )} \\
& =\frac{1}{\pi }\sum_{k=0}^{\infty }\frac{(-z)^{k}\Gamma (\alpha (k+1))\sin
(\pi (k+1)\alpha )}{k!},
\end{align*}%
for $0<\alpha <1$, is an entire function; it has the following properties:

\begin{itemize}
\item[(a)] $\phi _{\alpha }(\theta )\geq 0$ for $\theta \geq 0$ and $%
\displaystyle\int_{0}^{+\infty }\phi _{\alpha }(\theta )d\theta =1$;

\item[(b)] $\displaystyle\int_{0}^{+\infty }\phi _{\alpha }(\theta )\theta
^{r}d\theta =\frac{\Gamma (1+r)}{\Gamma (1+\alpha r)}$ for $r>-1$;

\item[(c)] $\displaystyle\int_{0}^{+\infty }\phi _{\alpha }(\theta
)e^{-z\theta }d\theta =E_{\alpha ,1}(-z)$, $z\in \mathbb{C}$;

\item[(d)] $\alpha \displaystyle\int_{0}^{+\infty }\theta \phi _{\alpha
}(\theta )e^{-z\theta }d\theta =E_{\alpha ,\alpha }(-z)$, $z\in \mathbb{C}$.
\end{itemize}

The operator $A=-\Delta $ with domain 
\begin{equation*}
D(A)=\{u\in C_{0}(\mathbb{R}^{N}):\Delta u\in C_{0}(\mathbb{R}^{N})\},
\end{equation*}%
generates, on $C_{0}(\mathbb{R}^{N})$, a semigroup $\{T(t)\}_{t\geq 0}$,
where 
\begin{equation*}
T(t)u_{0}\left( x\right) =\int_{\mathbb{R}^{N}}G(t,x-y)u_{0}(y)dy,\qquad
G(t,x)=\frac{1}{(4\pi t)^{N/2}}e^{-|x|^{2}/4t};
\end{equation*}%
it is analytic and contractive on $L^{q}(\mathbb{R}^{N})$ \cite{CazenaveHar}
and, for $t>0$, $x\in \mathbb{R}^{N}$, it satisfies 
\begin{equation}
\Vert T(t)u_{0}\Vert _{L^{p}(\mathbb{R}^{N})}\leq (4\pi t)^{-\frac{N}{2}%
(1/q-1/p)}\Vert u_{0}\Vert _{L^{q}(\mathbb{R}^{N})},  \label{Sehs}
\end{equation}%
for $1\leq q\leq p\leq +\infty $.

Let the operators $P_{\alpha }(t)$ and $S_{\alpha }(t)$ be defined by 
\begin{eqnarray}
P_{\alpha }(t)u_{0}=\int_{0}^{\infty }\phi _{\alpha }(\theta )T(t^{\alpha
}\theta )u_{0}d\theta ,\quad t\geq 0,\text{ }u_{0}\in C_{0}(\mathbb{R}^{N}),
\label{palpha} \\
S_{\alpha }(t)u_{0}=\alpha \int_{0}^{\infty }\theta \phi _{\alpha }(\theta
)T(t^{\alpha }\theta )u_{0}d\theta ,\quad t\geq 0,\text{ }u_{0}\in C_{0}(%
\mathbb{R}^{N}).  \label{salpha}
\end{eqnarray}%
The operators $P_{\alpha }(t)$ and $S_{\alpha }(t)$ acting on the space $%
C_{0}(\mathbb{R}^{N})$ into itself, see \cite[Lemma 2.3 , Lemma 2.4. ]%
{ZhangQuan}

Consider the problem 
\begin{eqnarray}
{}^{C}D_{0|t}^{\alpha }u-\Delta u &=&f(t,x),\quad t>0,\;x\in \mathbb{R}^{N},
\notag \\
u(0,x) &=&u_{0}(x),\quad x\in \mathbb{R}^{N},  \label{hlfde}
\end{eqnarray}%
where $u_{0}\in C_{0}(\mathbb{R}^{N})$ and $f\in L^{1}((0,T),C_{0}(\mathbb{R}%
^{N}))$. If $u$ is a solution of (\ref{hlfde}), then by \cite{ZhangQuan}, it
satisfies 
\begin{equation*}
u(t,x)=P_{\alpha }(t)u_{0}\left( x\right) +\int_{0}^{t}(t-s)^{\alpha
-1}S_{\alpha }(t-s)f(s,x)ds.
\end{equation*}%
The following lemmas play an important role in obtaining the results of this
paper; their proofs are obtained by combining smoothing effect of the heat
semigroup property \eqref{Sehs} with formulas \eqref{palpha} and %
\eqref{salpha} (see \cite{ZhangQuan}).

\begin{lemma}
\label{LpqeP} The operator $\{ P_{\alpha }( t) \}_{t>0} $ has the following
properties:

\begin{itemize}
\item[(a)] If $u_0\geq 0,u_0\not\equiv 0$, then $P_{\alpha }( t) u_0>0$ and $%
\| P_{\alpha }( t) u_0\| _{L^{1}( \mathbb{R}^N)} =\| u_0\| _{L^{1}( \mathbb{R%
}^N) }$;

\item[(b)] If $p\leq q\leq +\infty $ and $1/r=1/p-1/q$, $1/r<2/N$, then 
\begin{equation*}
\| P_{\alpha }( t) u_0\| _{L^q( \mathbb{R}^N) } \leq ( 4\pi t^{\alpha }) ^{-%
\frac{N}{2r}} \frac{\Gamma ( 1-N/( 2r) ) }{\Gamma ( 1-\alpha N/(2r) ) }\|
u_0\| _{L^p ( \mathbb{R}^N) }.
\end{equation*}
\end{itemize}
\end{lemma}

\begin{lemma}
\label{LpqeS} For the operator family $\{ S_{\alpha }( t)\} _{t>0}$, we have
the following estimates:

\begin{itemize}
\item[(a)] If $u_0\geq 0$ and $u_0\not\equiv 0$, then $S_{\alpha }( t) u_0>0$
and 
\begin{equation*}
\|S_{\alpha }( t) u_0\| _{L^{1}( \mathbb{R}^N) }=\frac{1}{\Gamma ( \alpha ) }%
\| u_0\| _{L^{1}( \mathbb{R}^N) };
\end{equation*}

\item[(b)] If $p\leq q\leq +\infty $ and $1/r=1/p-1/q$, $1/r<4/N$, then 
\begin{equation*}
\| S_{\alpha }( t) u_0\| _{L^q( \mathbb{R}^N) } \leq \alpha ( 4\pi t^{\alpha
}) ^{-\frac{N}{2r}} \frac{\Gamma ( 1-N/( 2r) ) }{\Gamma ( 1+\alpha -\alpha
N/( 2r) ) }\| u_0\|_{L^p( \mathbb{R}^N) }.
\end{equation*}
\end{itemize}
\end{lemma}

\begin{lemma}
\label{poldec} Let $l\geq 1$, and let the function $f(t, x) $ satisfy 
\begin{equation*}
\| f(t,\cdot) \| _l\leq C_1, \quad 0\leq t\leq 1,\quad \| f(t,\cdot) \|
_l\leq C_2t^{-\alpha }, \quad t>1,
\end{equation*}
for some positive constants $C_1$, $C_2$ and $\alpha$. Then 
\begin{equation*}
\| f(t,\cdot) \| _l\leq \max \{ C_1,C_2\} ( 1+t) ^{-\beta },\quad \text{for
any }\beta ,\quad 0<\beta \leq \alpha \text{ and }t\geq 0.
\end{equation*}
\end{lemma}

\begin{proof}
For $0\leq t\leq 1$, we have $\left\Vert f\left( t,.\right) \right\Vert
_{l}\leq C_{1}\leq C_{1}2^{\alpha }\left( 1+t\right) ^{-\alpha }$, so%
\begin{equation*}
\left\Vert f\left( t,.\right) \right\Vert _{l}\leq K\left( 1+t\right)
^{-\beta }\text{,}
\end{equation*}
for some positive constant $K>0$, and for any $0<\beta \leq \alpha $.

When $t\geq 1$, it follows from $\left\Vert f\left( t,.\right) \right\Vert
_{l}\leq C_{2}t^{-\alpha }$ that there is a constant $K^{\prime }>0$, such
that $\left\Vert f\left( t,.\right) \right\Vert _{l}\leq K^{\prime }\left(
1+t\right) ^{-\alpha }$, and so for any $0<\beta \leq \alpha $ and any $%
t\geq 1$, we have%
\begin{equation*}
\left\Vert f\left( t,.\right) \right\Vert _{l}\leq K^{\prime }\left(
1+t\right) ^{-\beta }\text{, for any }0<\beta \leq \alpha .
\end{equation*}%
Therefore%
\begin{equation*}
\left\Vert f\left( t,.\right) \right\Vert _{l}\leq \max \left\{ K,K^{\prime
}\right\} \left( 1+t\right) ^{-\beta }\text{, }\forall 0<\beta \leq \alpha 
\text{ and }t\geq 0\text{.}
\end{equation*}
\end{proof}

\section{Main results}

In this section, we state our main result. First, we present the definition
of a mild solution of problem \eqref{sys1}-\eqref{initdat}.

\begin{definition}
\textrm{Let $(u_{0},v_{0})\in C_{0}(\mathbb{R}^{N})\times C_{0}(\mathbb{R}%
^{N})$, $0<\gamma _{1},\gamma _{2}<1$, $p,q\geq 1$ and $T>0$. We say that $%
(u,v)\in C([0,T];C_{0}(\mathbb{R}^{N})\times C_{0}(\mathbb{R}^{N}))$ a mild
solution of system \eqref{sys1}-\eqref{initdat} if $(u,v)$ satisfies the
following integral equations 
\begin{eqnarray}
u(t,x) &=&P_{\gamma _{1}}(t)u_{0}+\int_{0}^{t}(t-\tau )^{\gamma
_{1}-1}S_{\gamma _{1}}(t-\tau )f(v(\tau ,\cdot ))d\tau ,  \notag \\
v(t,x) &=&P_{\gamma _{2}}(t)v_{0}+\int_{0}^{t}(t-\tau )^{\gamma
_{2}-1}S_{\gamma _{2}}(t-\tau )g(u(\tau ,.))d\tau .  \label{ms}
\end{eqnarray}%
}
\end{definition}

Using the results in \cite[Theorem 3.2]{ZhangQuan} and \cite[Theorem 3.2]%
{ZhangQuanLi}, the local solvability and uniqueness of \eqref{sys1}-%
\eqref{initdat} can be established.

\begin{proposition}[Existence of a local mild solution]
\label{FD} Given $u_0, v_0 \in C_0( \mathbb{R}^N) $, $0<\gamma _1$$, \gamma
_2<1$, $p$, $q\geq 1$, there exist a maximal time $T_{\mathrm{max}}>0$ and a
unique mild solution $( u,v) \in C( [ 0,T_{\mathrm{max}}] ;C_0( \mathbb{R}%
^N) \times C_0( \mathbb{R} ^N) ) $ to problem \eqref{sys1}-\eqref{initdat},
such that either

\begin{itemize}
\item[(i)] $T_{\mathrm{max}}=\infty $ (the solution is global), or

\item[(ii)] $T_{\mathrm{max}}<\infty $ and $\lim_{t\to T_{\mathrm{max}}} (
\| u(t)\| _{\infty }+\| v(t)\|_{\infty }) =\infty $ (the solution blows up
in a finite time).

\noindent If, in addition, $u_0\geq 0$, $v_0\geq 0$, $u_0$, $v_0\not\equiv 0$%
, then $u( t) >0$, $v( t) >0$ and $u( t) \geq P_{\gamma _1}( t) u_0$, $v( t)
\geq P_{\gamma _2}( t) v_0$ for $t\in ( 0,T_{\mathrm{max}}) $.
\end{itemize}

Moreover, if $( u_0,v_0) \in L^{1}( \mathbb{R} ^N) \times L^{1}( \mathbb{R}%
^N) $, then for any $s_1$, $s_2\in ( 1,+\infty ) $, $( u,v) \in C( [0,T_{%
\mathrm{max}}] ; L^{s_1}( \mathbb{R}^N) \times L^{s_2}( \mathbb{R}^N) ) $.
\end{proposition}

Now, we state the first main result of this section concerning the existence
of a global solution and large time behavior of solutions of \eqref{sys1}-%
\eqref{initdat}.

\begin{theorem}[Existence of a global mild solution]
\label{GELT} Let $N\geq 1$, let $q\geq p\geq 1$, be such that $pq>1$, let $%
\left( f(v),g(u)\right) =\left( \pm |v|^{p-1}v,\pm |u|^{q-1}u\right) $, or $%
\left( \pm |v|^{p},\pm |u|^{q}\right) $, and let $0<\gamma _{1}\leq \gamma
_{2}<1$. If 
\begin{equation}
\frac{N}{2}\geq \frac{\left( \gamma _{2}-\gamma _{1}\right) pq+q\gamma
_{2}+\gamma _{1}}{\gamma _{1}\left( pq-1\right) },  \label{critdimension}
\end{equation}%
then, for 
\begin{equation*}
\Vert u_{0}\Vert _{1}+\Vert u_{0}\Vert _{\infty }+\Vert v_{0}\Vert
_{1}+\Vert v_{0}\Vert _{\infty }\leq \varepsilon _{0},
\end{equation*}%
with some $\varepsilon _{0}>0$, there exist $s_{1}>q,$ $s_{2}>p$ such that
problem \eqref{sys1}-\eqref{initdat} admits a global mild solution with 
\begin{equation*}
u\in L^{\infty }([0,\infty ),L^{\infty }(\mathbb{R}^{N}))\cap L^{\infty
}([0,\infty ),L^{s_{1}}(\mathbb{R}^{N})),\text{ }v\in L^{\infty }([0,\infty
),L^{\infty }(\mathbb{R}^{N}))\cap L^{\infty }([0,\infty ),L^{s_{2}}(\mathbb{%
R}^{N})).
\end{equation*}%
Furthermore, for any $\delta >0$,%
\begin{equation*}
\max \left\{ 1-\frac{(pq-1)}{\gamma _{2}q(p+1)},1-\frac{\gamma _{1}\left(
pq-1\right) }{\gamma _{2}(p+1)},1-\frac{\left( pq-1\right) }{q+1}\right\}
<\delta <\min \big\{1,\frac{N(pq-1)}{2q(p+1)}\big\},
\end{equation*}
\begin{equation*}
\Vert u(t)\Vert _{s_{1}}\leq C(t+1)^{-\frac{(1-\delta )(\gamma _{1}+p\gamma
_{2})}{pq-1}},\Vert v(t)\Vert _{s_{2}}\leq C(t+1)^{-\frac{(1-\delta )(\gamma
_{2}+q\gamma _{1})}{pq-1}},\quad t\geq 0.
\end{equation*}%
If, in addition, 
\begin{equation*}
\frac{pN}{2s_{2}}<1\text{ and }\frac{qN}{2s_{1}}<1,
\end{equation*}%
or 
\begin{equation*}
N>2,pN/(2s_{2})<1\text{ and }qN/(2s_{1})\geq 1,
\end{equation*}%
or 
\begin{equation*}
N>2,\quad qN/(2s_{1})\geq 1,\quad pN/(2s_{2})\geq 1\quad \text{and }q\geq p>1
\end{equation*}%
with 
\begin{equation*}
\max \left\{ \frac{q+1}{pq(p+1)},\frac{pq-1}{pq(p+1)},\gamma _{2}/p,\sqrt{%
\frac{\gamma _{2}}{pq}}\right\} <\gamma _{1}\leq \gamma _{2}<1,
\end{equation*}%
then $u,v\in L^{\infty }([0,\infty ),L^{\infty }(\mathbb{R}^{N}))$, 
\begin{equation*}
\Vert u(t)\Vert _{\infty }\leq C(t+1)^{-\tilde{\sigma}},\quad \Vert
v(t)\Vert _{\infty }\leq C(t+1)^{-\hat{\sigma}},\quad t\geq 0,
\end{equation*}%
for some constants $\tilde{\sigma}>0$ and $\hat{\sigma}>0$.
\end{theorem}

\begin{definition}[Weak solution]
\label{Weaks} \textrm{Let $u_{0},v_{0}\in L_{loc}^{\infty }(\mathbb{R}^{N})$%
, $T>0$. We say that 
\begin{equation*}
(u,v)\in L^{q}((0,T),L_{loc}^{\infty }(\mathbb{R}^{N}))\times
L^{p}((0,T),L_{loc}^{\infty }(\mathbb{R}^{N}))
\end{equation*}%
is a weak solution of \eqref{sys1}-\eqref{initdat} if 
\begin{eqnarray*}
\int_{0}^{T}\int_{\mathbb{R}^{N}}(|v|^{p}\varphi +u_{0}D_{t|T}^{\gamma
_{1}}\varphi )\,dx\,dt=\int_{0}^{T}\int_{\mathbb{R}^{N}}u(-\Delta \varphi
)\,dx\,dt+\int_{0}^{T}\int_{\mathbb{R}^{N}}uD_{t|T}^{\gamma _{1}}\varphi
\,dx\,dt, \\
\int_{0}^{T}\int_{\mathbb{R}^{N}}(|u|^{q}\varphi +v_{0}D_{t|T}^{\gamma
_{2}}\varphi )\,dx\,dt=\int_{0}^{T}\int_{\mathbb{R}^{N}}v(-\Delta \varphi
)\,dx\,dt+\int_{0}^{T}\int_{\mathbb{R}^{N}}vD_{t|T}^{\gamma _{2}}\varphi
\,dx\,dt,
\end{eqnarray*}%
for every $\varphi \in C_{t,x}^{1,2}([0,T]\times \mathbb{R}^{N})$ such that
supp$_{x}\varphi \mathrm{\subset \subset }$}$\mathbb{R}^{N}$\textrm{\ and $%
\varphi (T,\cdot )=0$.}
\end{definition}

Similar to the proof in \cite{ZhangQuan}, we can easily obtain the following
lemma asserting that the mild solution is the weak solution.

\begin{lemma}
Assume $u_0,v_0\in C_0( \mathbb{R}^N) $, and let $(u, v) \in C([0,T], C_0( 
\mathbb{R}^N) \times C_0(\mathbb{R}^N) $ be a mild solution of \eqref{sys1}-%
\eqref{initdat}, then $( u,v) $ is a weak solution of \eqref{sys1}-%
\eqref{initdat}.
\end{lemma}

Our next result concerns the blowing-up of solutions of \eqref{sys1}-%
\eqref{initdat}.

\begin{theorem}[Blow-up of mild solutions]
\label{NEG} Let $N\geq 1$, $p>1$, $q>1$, $0<\gamma _{1},\gamma _{2}<1,$ let $%
\left( f(v),g(u)\right) =\left( |v|^{p},|u|^{q}\right) $, let $u_{0},$ $%
v_{0}\in C_{0}(\mathbb{R}^{N})$, $u_{0}\geq 0$, $v_{0}\geq 0$, $%
u_{0}\not\equiv 0$ and $v_{0}\not\equiv 0$. If 
\begin{equation*}
\frac{N}{2}<\min \big\{\frac{(p\gamma _{2}+\gamma _{1})}{\gamma _{1}(pq-1)},%
\frac{(q\gamma _{1}+\gamma _{2})}{\gamma _{1}(pq-1)},\frac{(pq(\gamma
_{1}-\gamma _{2})+q\gamma _{1}+\gamma _{2})}{\gamma _{1}(pq-1)},\;\frac{(p+1)%
}{pq-1}\big\}
\end{equation*}%
or 
\begin{equation*}
\frac{N}{2}<\min \big\{\frac{(q\gamma _{1}+\gamma _{2})}{\gamma _{2}(pq-1)},%
\frac{(p\gamma _{2}+\gamma _{1})}{\gamma _{2}(pq-1)},\frac{(pq(\gamma
_{2}-\gamma _{1})+p\gamma _{2}+\gamma _{1})}{\gamma _{2}(pq-1)},\frac{(q+1)}{%
pq-1}\big\},
\end{equation*}%
then the mild solution $(u,v)$ of \eqref{sys1}-\eqref{initdat} blows up in a
finite time.

Also if $p=1$ and $1<q<1+\frac{2}{N}$, or $1<p<1+\frac{2}{N}$ and $q=1$,
then the solution blows-up in a finite time.
\end{theorem}

A result of blowing-up solutions can be obtained via differential
inequalities. Let 
\begin{equation*}
\chi (x)=\left( \int_{\mathbb{R}^{N}}e^{-\sqrt{N^{2}+|x|^{2}}}\,dx\right)
^{-1}e^{-\sqrt{N^{2}+|x|^{2}}},\quad x\in \mathbb{R}^{N},
\end{equation*}%
which satisfies 
\begin{equation*}
\int_{\mathbb{R}^{N}}\chi (x)\,dx=1.
\end{equation*}%
In the next theorem, we take $f(v)=\left\vert v\right\vert ^{p}$ and $%
g(u)=\left\vert u\right\vert ^{q}$.

\begin{theorem}
\label{Blwr} Let $\gamma _{1}=\gamma _{2}=\gamma \in (0,1)$, $u_{0}$, $%
v_{0}\in C_{0}(\mathbb{R}^{N})$ and $u_{0}$, $v_{0}\geq 0$. Let $p>1,$ $q>1$
such that $p\leq q$ and let $\left( f(v),g(u)\right) =\left(
|v|^{p},|u|^{q}\right) $. If 
\begin{equation*}
Z_{0}:=\int_{\mathbb{R}^{N}}(u_{0}(x)+v_{0}(x))\chi (x)dx> 2^{\frac{p}{p-1}},
\end{equation*}%
then the solution of problem \eqref{sys1}-\eqref{initdat} blows-up in a
finite time. Moreover, we have estimate of the time blowing up $\bar{t}%
_{\ast \ast }\leq \left[ \frac{\ln \left( 1-2^{p}Z_{0}^{1-p}\right) }{2(1-p)}%
\Gamma \left( \gamma +1\right) \right] ^{1/\gamma }.$
\end{theorem}

The next lemma plays an important role in establishing lower solution for
Caputo fractional differential equation

\begin{lemma}[{\protect\cite[Lemma 3.1]{Vats} }]
\label{vatsala lem}Let $u=u(t)$ is a solution of the ordinary differential
equation%
\begin{equation}
\frac{du}{dt}=F(u),\text{ }u\left( 0\right) =u_{0},  \label{ODE}
\end{equation}%
where $F$ is a function of $u$ such that $F(0)\geq 0,$ $F(u)>0$, $%
F_{u}\left( u\right) \geq 0$ for $u\geq 0$ then $v(t)=u(\bar{t})$ is a lower
solution of a Caputo fractional differential equation%
\begin{equation}
^{C}D_{0|t}^{\alpha }u(t)=F(u),\text{ }u\left( 0\right) =u_{0},
\label{focde}
\end{equation}%
where $\bar{t}=\frac{t^{\alpha }}{\Gamma \left( \alpha +1\right) }$. That
means 
\begin{equation*}
^{C}D_{0|t}^{\alpha }v(t)\leq F(v),v\left( 0\right) \leq u_{0}.
\end{equation*}
\end{lemma}

\section{Proofs of main results}

\begin{proof}[Proof of Theorem \protect\ref{GELT}]
We proceed in three steps. \smallskip

\noindent \textbf{Step 1:} Global existence for $(u,v)$ in $L^{s_{1}}(%
\mathbb{R}^{N})\times L^{s_{2}}(\mathbb{R}^{N})$. Since $q\geq p\geq 1,$ $%
pq>1,$ $0<\gamma _{1}\leq \gamma _{2}<1$ and 
\begin{equation*}
\frac{N}{2}\geq \frac{\left( \gamma _{2}-\gamma _{1}\right) pq+q\gamma
_{2}+\gamma _{1}}{\gamma _{1}\left( pq-1\right) }>\frac{\gamma _{2}q\left(
p+1\right) -\gamma _{1}q\left( pq-1\right) }{\gamma _{2}\left( pq-1\right) },
\end{equation*}%
we have $\frac{N(pq-1)}{2q(p+1)}>\frac{\gamma _{2}\left( p+1\right) -\gamma
_{1}\left( pq-1\right) }{\gamma _{2}(p+1)}=1-\frac{\gamma _{1}\left(
pq-1\right) }{\gamma _{2}(p+1)}.$

Note also, from 
\begin{equation*}
\frac{N}{2}\geq \frac{\left( \gamma _{2}-\gamma _{1}\right) pq+q\gamma
_{2}+\gamma _{1}}{\gamma _{1}\left( pq-1\right) }>\frac{q+1}{pq-1},
\end{equation*}%
we get 
\begin{equation*}
\frac{N(pq-1)}{2q(p+1)}>\frac{q+1}{pq-1}\times \frac{(pq-1)}{q(p+1)}=\frac{%
q+1}{q(p+1)}>1-\frac{\left( pq-1\right) }{\gamma _{2}q(p+1)},
\end{equation*}%
and%
\begin{equation*}
\frac{N(pq-1)}{2q(p+1)}>\frac{q+1}{q(p+1)}>1-\frac{\left( pq-1\right) }{q+1}.
\end{equation*}

From these facts, we can choose $\delta >0$ such that 
\begin{equation}
\max \left\{ 1-\frac{\left( pq-1\right) }{\gamma _{2}q(p+1)},1-\frac{\gamma
_{1}\left( pq-1\right) }{\gamma _{2}(p+1)},1-\frac{\left( pq-1\right) }{q+1}%
\right\} <\delta <\min \left\{ 1,\frac{N(pq-1)}{2q(p+1)}\right\} \text{.}
\label{delta}
\end{equation}%
We set

\begin{eqnarray}
r_{1} &=&\frac{N\gamma _{1}(pq-1)}{2[\gamma _{1}(1+\delta p)+\gamma
_{2}p(1-\delta )]},\quad r_{2}=\frac{N\gamma _{2}(pq-1)}{2[\gamma
_{2}(1+\delta q)+\gamma _{1}q(1-\delta )]},  \notag \\
\frac{1}{s_{1}} &=&\frac{2\delta }{N}\frac{p+1}{pq-1},\quad \frac{1}{s_{2}}=%
\frac{2\delta }{N}\frac{q+1}{pq-1},  \label{S_1,S_2} \\
\sigma _{1} &=&\frac{(1-\delta )(\gamma _{1}+\gamma _{2}p)}{pq-1},\quad
\sigma _{2}=\frac{(1-\delta )(\gamma _{2}+\gamma _{1}q)}{pq-1}.  \notag
\end{eqnarray}%
Clearly, we have 
\begin{gather*}
\frac{1}{r_{1}}=\frac{2}{N\gamma _{1}}\frac{(1-\delta )(\gamma _{1}+\gamma
_{2}p)}{pq-1}+\frac{2\delta }{N}\frac{(p+1)}{pq-1}, \\
\frac{1}{r_{2}}=\frac{2}{N\gamma _{2}}\frac{(1-\delta )(\gamma _{2}+\gamma
_{1}q)}{pq-1}+\frac{2\delta }{N}\frac{(q+1)}{pq-1}.
\end{gather*}%
It is easy to check that $s_{1}>q$, $s_{2}>p$, $ps_{1}>s_{2}$, $qs_{2}>s_{1}$%
, $s_{1}>r_{1}>1$, $s_{2}>r_{2}>1$, 
\begin{equation*}
\frac{N}{2}\gamma _{1}\big(\frac{1}{r_{1}}-\frac{1}{s_{1}}\big)q<1,\quad 
\frac{N}{2}\gamma _{2}\big(\frac{1}{r_{2}}-\frac{1}{s_{2}}\big)p<1,\quad 
\frac{N}{2}(\frac{p}{s_{2}}-\frac{1}{s_{1}})=\delta =\frac{N}{2}(\frac{q}{%
s_{1}}-\frac{1}{s_{2}}),
\end{equation*}%
$p\sigma _{2}<1$, and $q\sigma _{1}<1$. From 
\begin{equation*}
\frac{pq(\gamma _{2}-1)+1+q\gamma _{1}}{[p\gamma _{2}+\gamma _{1}]q}=\frac{%
q(p\gamma _{2}+\gamma _{1})-\left( pq-1\right) }{[p\gamma _{2}+\gamma _{1}]q}%
<1-\frac{\left( pq-1\right) }{\gamma _{2}q(p+1)}<\delta
\end{equation*}%
we obtain $\delta >\frac{pq\left( \gamma _{2}-1\right) +q\gamma _{1}+1}{%
\left( \gamma _{1}+p\gamma _{2}\right) q}$ which is equivalent to%
\begin{equation*}
\big(\gamma _{1}-\frac{N}{2}\gamma _{1}(\frac{p}{s_{2}}-\frac{1}{s_{1}}%
)-p\sigma _{2}\big)q>-1.
\end{equation*}%
In fact, since $\delta =\frac{N}{2}(\frac{p}{s_{2}}-\frac{1}{s_{1}})$, the
last inequality gives 
\begin{equation*}
\left( \gamma _{1}-\delta \gamma _{1}-p\sigma _{2}\right) q>-1,
\end{equation*}%
using definition of $\sigma _{2}=\frac{(1-\delta )(\gamma _{2}+\gamma _{1}q)%
}{pq-1},$ we get%
\begin{equation*}
\left( \gamma _{1}-\delta \gamma _{1}-p\frac{(1-\delta )(\gamma _{2}+\gamma
_{1}q)}{pq-1}\right) q>-1,
\end{equation*}%
so, we obtain%
\begin{equation*}
(1-\delta )\left( \gamma _{1}-p\frac{(\gamma _{2}+\gamma _{1}q)}{pq-1}%
\right) q>-1.
\end{equation*}%
Therefore%
\begin{equation*}
(1-\delta )\left( \frac{\left( pq-1\right) \gamma _{1}-p(\gamma _{2}+\gamma
_{1}q)}{pq-1}\right) q>-1.
\end{equation*}%
By simplification, we obtain%
\begin{equation*}
(\delta -1)\left( \frac{\gamma _{1}+p\gamma _{2}}{pq-1}\right) q>-1,
\end{equation*}%
or%
\begin{equation*}
\delta \left( \frac{\gamma _{1}+p\gamma _{2}}{pq-1}\right) q>\left( \frac{%
\gamma _{1}+p\gamma _{2}}{pq-1}\right) q-1.
\end{equation*}%
Thus%
\begin{equation*}
\delta >1-\frac{pq-1}{\left( \gamma _{1}+p\gamma _{2}\right) q}=\frac{%
pq\left( \gamma _{2}-1\right) +q\gamma _{1}+1}{\left( \gamma _{1}+p\gamma
_{2}\right) q}.
\end{equation*}%
Similarly, we have%
\begin{equation*}
\big(\gamma _{2}-\frac{N}{2}\gamma _{2}(\frac{q}{s_{1}}-\frac{1}{s_{2}}%
)-q\sigma _{1}\big)p>-1.
\end{equation*}%
equivalent to $\delta >\frac{pq\left( \gamma _{1}-1\right) +p\gamma _{2}+1}{%
\left( \gamma _{2}+q\gamma _{1}\right) p}.$

Let $(u_{0},v_{0})\in C_{0}(\mathbb{R}^{N})\times C_{0}(\mathbb{R}^{N})\cap
L^{r_{1}}(\mathbb{R}^{N})\times L^{r_{2}}(\mathbb{R}^{N})$, and let 
\begin{equation*}
(u,v)\in C([0,T_{\mathrm{max}});C_{0}(\mathbb{R}^{N})\cap L^{s_{1}}(\mathbb{R%
}^{N}))\times C([0,T_{\mathrm{max}});C_{0}(\mathbb{R}^{N})\cap L^{s_{2}}(%
\mathbb{R}^{N})).
\end{equation*}%
For $t\in \lbrack 0,T_{\mathrm{max}})$, from (\ref{ms}), we have 
\begin{eqnarray*}
\Vert u(t)\Vert _{s_{1}} &\leq &\Vert P_{\gamma _{1}}(t)u_{0}\Vert
_{s_{1}}+\left\Vert \int_{0}^{t}(t-\tau )^{\gamma _{1}-1}S_{\gamma
_{1}}(t-\tau )|v(\tau )|^{p}d\tau \right\Vert _{s_{1}} \\
&=&\Vert P_{\gamma _{1}}(t)u_{0}\Vert _{s_{1}}+\left[ \int_{\mathbb{R}%
^{N}}\left\vert \int_{0}^{t}(t-\tau )^{\gamma _{1}-1}S_{\gamma _{1}}(t-\tau
)|v(\tau )|^{p}d\tau \right\vert ^{s_{1}}dx\right] ^{1/s_{1}}.
\end{eqnarray*}%
We have 
\begin{eqnarray*}
&&\left[ \int_{\mathbb{R}^{N}}\left\vert \int_{0}^{t}(t-\tau )^{\gamma
_{1}-1}S_{\gamma _{1}}(t-\tau )|v(\tau )|^{p}d\tau \right\vert ^{s_{1}}dx%
\right] ^{1/s_{1}} \\
&\leq &\left[ \int_{\mathbb{R}^{N}}\left( \int_{0}^{t}(t-\tau )^{\gamma
_{1}-1}\left\vert S_{\gamma _{1}}(t-\tau )|v(\tau )|^{p}d\tau \right\vert
\right) ^{s_{1}}dx\right] ^{1/s_{1}}.
\end{eqnarray*}%
Making use of Minkowski's integral inequality, we get%
\begin{eqnarray*}
&&\left[ \int_{\mathbb{R}^{N}}\left( \int_{0}^{t}(t-\tau )^{\gamma
_{1}-1}\left\vert S_{\gamma _{1}}(t-\tau )|v(\tau )|^{p}d\tau \right\vert
\right) ^{s_{1}}dx\right] ^{1/s_{1}} \\
&\leq &\int_{0}^{t}\left( \int_{\mathbb{R}^{N}}(t-\tau )^{s_{1}\left( \gamma
_{1}-1\right) }\left\vert S_{\gamma _{1}}(t-\tau )|v(\tau )|^{p}\right\vert
^{s_{1}}dx\right) ^{1/s_{1}}d\tau \\
&=&\int_{0}^{t}(t-\tau )^{\gamma _{1}-1}\left( \int_{\mathbb{R}%
^{N}}\left\vert S_{\gamma _{1}}(t-\tau )|v(\tau )|^{p}\right\vert
^{s_{1}}dx\right) ^{1/s_{1}}d\tau \\
&=&\int_{0}^{t}(t-\tau )^{\gamma _{1}-1}\Vert S_{\gamma _{1}}(t-\tau
)|v(\tau )|^{p}\Vert _{s_{1}}d\tau .
\end{eqnarray*}%
Hence, we obtain for $t\in \lbrack 0,T_{\mathrm{max}})$,%
\begin{eqnarray}
\Vert u(t)\Vert _{s_{1}} &\leq &\Vert P_{\gamma _{1}}(t)u_{0}\Vert
_{s_{1}}+\int_{0}^{t}(t-\tau )^{\gamma _{1}-1}\Vert S_{\gamma _{1}}(t-\tau
)|v(\tau )|^{p}\Vert _{s_{1}}d\tau ,  \label{lso1} \\
\Vert v(t)\Vert _{s_{2}} &\leq &\Vert P_{\gamma _{2}}(t)v_{0}\Vert
_{s_{2}}+\int_{0}^{t}(t-\tau )^{\gamma _{2}-1}\Vert S_{\gamma _{2}}(t-\tau
)|u(\tau )|^{q}\Vert _{s_{2}}d\tau .  \label{lso2}
\end{eqnarray}%
Applying lemmas \ref{LpqeS} and \ref{LpqeP}, we obtain 
\begin{eqnarray}
\Vert u(t)\Vert _{s_{1}} &\leq &\Vert u_{0}\Vert _{r_{1}}t^{-\sigma
_{1}}+C\int_{0}^{t}(t-\tau )^{\gamma _{1}-1}(t-\tau )^{-\frac{N}{2}\gamma
_{1}(\frac{p}{s_{2}}-\frac{1}{s_{1}})}\Vert v(\tau )\Vert _{s_{2}}^{p}d\tau ,
\label{lso3} \\
\Vert v(t)\Vert _{s_{2}} &\leq &\Vert v_{0}\Vert _{r_{2}}t^{-\sigma
_{2}}+C\int_{0}^{t}(t-\tau )^{\gamma _{2}-1}(t-\tau )^{-\frac{N}{2}\gamma
_{2}(\frac{q}{s_{1}}-\frac{1}{s_{2}})}\Vert u(\tau )\Vert _{s_{1}}^{q}d\tau .
\label{lso4}
\end{eqnarray}%
By using \eqref{lso4} in \eqref{lso3}, we obtain 
\begin{eqnarray*}
\Vert u(t)\Vert _{s_{1}} &\leq &\Vert u_{0}\Vert _{r_{1}}t^{-\sigma _{1}}+C%
\displaystyle\int_{0}^{t}(t-\tau )^{\gamma _{1}-1}(t-\tau )^{-\frac{N}{2}%
\gamma _{1}(\frac{p}{s_{2}}-\frac{1}{s_{1}})}d\tau \\
&&\quad \times \Big(\Vert v_{0}\Vert _{r_{2}}t^{-\sigma _{2}}+C\displaystyle%
\int_{0}^{t}(t-\tau )^{\gamma _{2}-1}(t-\tau )^{-\frac{N}{2}\gamma _{2}(%
\frac{q}{s_{1}}-\frac{1}{s_{2}})}\Vert u\Vert _{s_{1}}^{q}d\tau \Big)^{p}.
\end{eqnarray*}%
Hence 
\begin{eqnarray}
\Vert u(t)\Vert _{s_{1}} &\leq &\Vert u_{0}\Vert _{r_{1}}t^{-\sigma _{1}}+C%
\displaystyle\int_{0}^{t}(t-\tau )^{\gamma _{1}-1-\frac{N}{2}\gamma _{1}(%
\frac{p}{s_{2}}-\frac{1}{s_{1}})}\tau ^{-p\sigma _{2}}d\tau \Vert v_{0}\Vert
_{r_{2}}^{p}  \notag \\
&&\quad +C\displaystyle\int_{0}^{t}(t-\tau )^{\gamma _{1}-1-\frac{N}{2}%
\gamma _{1}(\frac{p}{s_{2}}-\frac{1}{s_{1}})}\tau ^{(\gamma _{2}-\frac{N}{2}%
\gamma _{2}(\frac{q}{s_{1}}-\frac{1}{s_{2}})-q\sigma _{1})p}  \notag \\
&&\quad \times \Big(\tau ^{^{\sigma _{1}}}\Vert u(\tau )\Vert _{s_{1}}\Big)%
^{pq}d\tau .  \label{lsom}
\end{eqnarray}%
Multiplying both sides of \eqref{lsom} by $t^{\sigma _{1}}$, where $\sigma
_{1}=\frac{(1-\delta )(\gamma _{1}+\gamma _{2}p)}{pq-1}$, we find that 
\begin{eqnarray}
t^{\sigma _{1}}\Vert u\Vert _{s_{1}} &\leq &\Vert u_{0}\Vert
_{r_{1}}+Ct^{\sigma _{1}}\int_{0}^{t}(t-\tau )^{\gamma _{1}-1-\frac{N}{2}%
\gamma _{1}(\frac{p}{s_{2}}-\frac{1}{s_{1}})}\tau ^{-p\sigma _{2}}d\tau
\Vert v_{0}\Vert _{r_{2}}^{p}  \notag \\
&&\quad +Ct^{\sigma _{1}}\int_{0}^{t}(t-\tau )^{\gamma _{1}-1-\frac{N}{2}%
\gamma _{1}(\frac{p}{s_{2}}-\frac{1}{s_{1}})}\tau ^{(\gamma _{2}-\frac{N}{2}%
\gamma _{2}(\frac{q}{s_{1}}-\frac{1}{s_{2}})-q\sigma _{1})p}  \notag \\
&&\quad \times \Big(\tau ^{^{\sigma _{1}}}\Vert u\Vert _{s_{1}}\Big)%
^{pq}d\tau .  \label{lso5}
\end{eqnarray}%
Since $\gamma _{1}-1-\frac{N}{2}\gamma _{1}(\frac{p}{s_{2}}-\frac{1}{s_{1}}%
)>-1$, and $\big(\gamma _{2}-\frac{N}{2}\gamma _{2}(\frac{q}{s_{1}}-\frac{1}{%
s_{2}})-q\sigma _{1}\big)p>-1$, we have 
\begin{eqnarray}
&&t^{\sigma _{1}}\Vert u\Vert _{s_{1}}  \notag \\
&\leq &\Vert u_{0}\Vert _{r_{1}}+Ct^{\sigma _{1}+\gamma _{1}-\frac{N}{2}%
\gamma _{1}(\frac{p}{s_{2}}-\frac{1}{s_{1}})-p\sigma _{2}}\Vert v_{0}\Vert
_{r_{2}}^{p}  \notag \\
&&\quad +Ct^{\sigma _{1}+\gamma _{1}-\frac{N}{2}\gamma _{1}(\frac{p}{s_{2}}-%
\frac{1}{s_{1}})+(\gamma _{2}-\frac{N}{2}\gamma _{2}(\frac{q}{s_{1}}-\frac{1%
}{s_{2}})-q\sigma _{1})p}\Big(\sup_{0\leq \tau <t}\tau ^{\sigma _{1}}\Vert
u(\tau )\Vert _{s_{1}}\Big)^{pq}.  \label{lso6}
\end{eqnarray}%
Note that%
\begin{equation*}
\begin{array}{l}
\sigma _{1}=\frac{N}{2}\gamma _{1}(\frac{1}{r_{1}}-\frac{1}{s_{1}}), \\ 
\sigma _{1}+\gamma _{1}-\frac{N}{2}\gamma _{1}\big(\frac{p}{s_{2}}-\frac{1}{%
s_{1}}\big)-p\sigma _{2}=0, \\ 
\sigma _{1}+\gamma _{1}-\frac{N}{2}\gamma _{1}\big(\frac{p}{s_{2}}-\frac{1}{%
s_{1}}\big)+\Big(\gamma _{2}-\frac{N}{2}\gamma _{2}\big(\frac{q}{s_{1}}-%
\frac{1}{s_{2}}\big)-q\sigma _{1}\Big)p=0, \\ 
\sigma _{1}+\gamma _{1}-\gamma _{1}\delta +(\gamma _{2}-\gamma _{2}\delta
-q\sigma _{1})p=0.%
\end{array}%
\end{equation*}%
Defining $h(t)=\sup_{0\leq \tau \leq t}\tau ^{^{\sigma _{1}}}\Vert u(\tau
)\Vert _{s_{1}}$, $t\in \lbrack 0,T_{\mathrm{max}})$, we deduce from %
\eqref{lso5} that 
\begin{equation}
h(t)\leq C(\Vert u_{0}\Vert _{r_{1}}+\Vert v_{0}\Vert _{r_{2}}^{p}+h(t)^{pq})
\label{lso7}
\end{equation}%
for any $t\in (0,T_{\mathrm{max}})$. Here $C$ is independent of $t$.\newline
Set 
\begin{equation*}
A:=\Vert u_{0}\Vert _{r_{1}}+\Vert v_{0}\Vert _{r_{2}}^{p}.
\end{equation*}%
Then, it follows by a continuity argument that for sufficiently small $u_{0}$
and $v_{0}$ such that $A<(2C)^{\frac{pq}{1-pq}}$ that 
\begin{equation}
h(t)\leq 2CA,\quad \text{for any }t\in \lbrack 0,T_{\mathrm{max}}).
\label{QQ1}
\end{equation}%
Indeed, otherwise, there exists $t_{0}\in (0,T_{\mathrm{max}})$ such that $%
h(t_{0})>2CA$; by the intermediate value theorem, since $h$ is continuous
and $h(0)=0$, there exists $t_{1}\in (0,t_{0})$ such that 
\begin{equation}
h(t_{1})=2CA.  \label{QQ2}
\end{equation}%
Using (\ref{QQ2}) in (\ref{lso7}), we obtain 
\begin{equation*}
h(t_{1})\leq C(A+h(t_{1})^{pq})=\left( \frac{h(t_{1})}{2}+Ch(t_{1})^{pq}%
\right) ,
\end{equation*}%
from which, we infer 
\begin{equation}
\frac{h(t_{1})}{2}\leq Ch(t_{1})^{pq},  \label{htone}
\end{equation}%
using (\ref{QQ2}) in (\ref{htone}), we get 
\begin{equation*}
CA\leq C\left( 2CA\right) ^{pq},
\end{equation*}%
so%
\begin{equation*}
A\leq \left( 2CA\right) ^{pq}=\left( 2C\right) ^{pq}A^{pq},
\end{equation*}%
then it yields%
\begin{equation*}
\left( 2C\right) ^{-pq}\leq A^{pq-1},
\end{equation*}

which is equivalent to 
\begin{equation*}
A\geq (2C)^{\frac{pq}{1-pq}}.
\end{equation*}%
This is contradiction with the choice of $A$. It then follows that $h(t)$
remains bounded in all time $t>0$ provided that $\Vert u_{0}\Vert _{r_{1}}$
and $\Vert v_{0}\Vert _{r_{2}}$ are small. Therefore 
\begin{equation}
t^{\sigma _{1}}\Vert u(t)\Vert _{s_{1}}\leq C,\quad \text{for any }t>0.
\label{lso8}
\end{equation}%
Similarly, we obtain 
\begin{equation}
t^{\sigma _{2}}\Vert v(t)\Vert _{s_{2}}\leq C,\quad \text{for any }t>0.
\label{lso9}
\end{equation}

\smallskip

\noindent \textbf{Step 2:} $L^{\infty }$-global existence estimates of $%
(u,v) $ in $L^{\infty }(\mathbb{R}^{N})\times L^{\infty }(\mathbb{R}^{N})$.

Let $s_{1}$, $s_{2}$ be the same as in \eqref{S_1,S_2}. Since $p\leq q$, we
have 
\begin{equation*}
\frac{Np}{2s_{2}}\leq \frac{Nq}{2s_{1}}.
\end{equation*}%
We further assume for some $\xi >q$, $w>p$, $k_{1}>0,$ $k_{2}>0$ that $%
u(t)\in L^{w}(\mathbb{R}^{N}),$ $v(t)\in L^{\xi }(\mathbb{R}^{N})$ and 
\begin{equation}
\Vert u(t)\Vert _{w}\leq C(1+t^{k_{1}}),\quad \Vert v(t)\Vert _{\xi }\leq
C(1+t^{k_{2}})\quad \text{for every }t\in \lbrack 0,T_{\mathrm{max}}).
\label{lso14}
\end{equation}%
Then, by applying Lemmas \ref{LpqeP} and \ref{LpqeS} again to \eqref{ms}, we
get 
\begin{gather}
\Vert u(t)\Vert _{\infty }\leq \Vert P_{\gamma _{1}}(t)u_{0}\Vert _{\infty
}+\int_{0}^{t}(t-\tau )^{\gamma _{1}-1-\frac{N\gamma _{1}p}{2\xi }}\Vert
v(\tau )\Vert _{\xi }^{p}d\tau ,  \label{lso15} \\
\Vert v(t)\Vert _{\infty }\leq \Vert P_{\gamma _{2}}(t)v_{0}\Vert _{\infty
}+\int_{0}^{t}(t-\tau )^{\gamma _{2}-1-\frac{N\gamma _{2}q}{2w}}\Vert u(\tau
)\Vert _{w}^{q}d\tau ,  \label{lso16}
\end{gather}%
for all $t\in \lbrack 0,T_{\mathrm{max}})$. Now, if one can find $\xi $ and $%
w$ such that 
\begin{equation}
\frac{Np}{2\xi }<1\text{ }\quad \text{or }\quad \frac{Nq}{2w}<1,
\label{test}
\end{equation}%
then the $L^{\infty }$-estimates of $(u,v)$ can be obtained. In fact, if $%
\frac{Np}{2\xi }<1$, in view of \eqref{lso14}, from \eqref{lso15} we have

\begin{eqnarray}
\| u( t) \| _{\infty } &\leq \| P_{\gamma _1}( t) u_0\| _{\infty
}+C\max_{\tau \in [ 0,t] }\| v( \tau ) \| _{\xi }^pt^{( 1-\frac{Np}{2\xi })
\gamma _1} \\
&\leq C(1+t^{( 1-\frac{Np}{2\xi }) \gamma _1+pk_2}) ,  \label{lso17}
\end{eqnarray}%
and by taking $w=\infty $ in \eqref{lso16}, we obtain

\begin{eqnarray}
\Vert v(t)\Vert _{\infty } &\leq &\Vert P_{\gamma _{2}}(t)v_{0}\Vert
_{\infty }+\displaystyle\int_{0}^{t}(t-\tau )^{\gamma _{2}-1}\Vert u(\tau
)\Vert _{\infty }^{q}d\tau \\
&\leq &\Vert P_{\gamma _{2}}(t)v_{0}\Vert _{\infty }+\displaystyle%
\int_{0}^{t}(t-\tau )^{\gamma _{2}-1}\Big(1+t^{(1-\frac{Np}{2\xi })\gamma
_{1}+pk_{2}}\Big)^{q}d\tau \\
&\leq &C\Big(1+t^{\gamma _{2}+[(1-\frac{Np}{2\xi })\gamma _{1}+pk_{2}]q}\Big)%
.  \label{lso18}
\end{eqnarray}%
These estimates show that $T_{\mathrm{max}}=\infty $ and 
\begin{equation}
u,v\in L_{\mathrm{loc}}^{\infty }([0,\infty );L^{\infty }(\mathbb{R}^{N})).
\label{test1}
\end{equation}%
In a similar way, we can deal with the case $\frac{Nq}{2w}<1$.

To find such appropriate $\xi $ and $w$, we note that if $\frac{Nq}{2s_{1}}%
<1 $ or $\frac{Np}{2s_{2}}<1$, then \eqref{lso17} and \eqref{lso18} hold by
taking $\xi =s_{1}$ or $w=s_{2}$. This is certainly the case if $N\leq 2$ as 
$s_{1}>q$ and $s_{2}>p$.

Thus it remains to deal with the case $N>2$, $\frac{Nq}{2s_{1}}\geq 1$ and $%
\frac{Np}{2s_{2}}\geq 1$. We will do this via an iterative process. Define $%
s_{1}^{\prime }=s_{1}$, $s_{1}^{\prime \prime }=s_{2}$. Since $s_{1}^{\prime
}>q$ and $s_{1}^{\prime \prime }>p$, using the H\"{o}lder inequality and
lemmas \ref{LpqeP},\ref{LpqeS}, we obtain from \eqref{lso1} and \eqref{lso2}
that 
\begin{gather*}
\Vert u(t)\Vert _{s_{2}^{\prime }}\leq \Vert P_{\gamma _{1}}(t)u_{0}\Vert
_{s_{2}^{\prime }}+\int_{0}^{t}(t-\tau )^{\gamma _{1}-1-\frac{N\gamma _{1}}{2%
}(\frac{p}{s_{2}^{\prime \prime }}-\frac{1}{s_{2}^{\prime }})}\Vert v(\tau
)\Vert _{s_{1}^{\prime \prime }}^{p}d\tau , \\
\Vert v(t)\Vert _{s_{2}^{\prime \prime }}\leq \Vert P_{\gamma
_{2}}(t)v_{0}\Vert _{s_{2}^{\prime \prime }}+\int_{0}^{t}(t-\tau )^{\gamma
_{2}-1-\frac{N\gamma _{2}}{2}(\frac{q}{s_{1}^{\prime }}-\frac{1}{%
s_{2}^{\prime \prime }})}\Vert u(\tau )\Vert _{s_{1}^{\prime }}^{q}d\tau ,
\end{gather*}%
where $s_{2}^{\prime }$ and $s_{2}^{\prime \prime }$ are such that 
\begin{equation*}
\frac{N}{2}\Big(\frac{p}{s_{1}^{\prime \prime }}-\frac{1}{s_{2}^{\prime }}%
\Big)<1,\quad \frac{N}{2}(\frac{q}{s_{1}^{\prime }}-\frac{1}{s_{2}^{\prime
\prime }})<1.
\end{equation*}%
This can be verified by taking 
\begin{equation*}
\frac{1}{s_{2}^{\prime }}=\frac{p}{s_{1}^{\prime \prime }}-\frac{2}{N}+\eta
,\quad \frac{1}{s_{2}^{\prime \prime }}=\frac{q}{s_{1}^{\prime }}-\frac{2}{N}%
+\eta ,
\end{equation*}%
where $0<\eta <\frac{2(1-\delta )}{N}$. Observe that 
\begin{equation}
\frac{1}{s_{1}^{\prime }}-\frac{1}{s_{2}^{\prime }}=\frac{2}{N}(1-\delta
)-\eta >0,\quad \frac{1}{s_{1}^{\prime \prime }}-\frac{1}{s_{2}^{\prime
\prime }}=\frac{2}{N}(1-\delta )-\eta >0,  \label{lso19}
\end{equation}%
and hence $s_{2}^{\prime }>s_{1}^{\prime }>q$ and $s_{2}^{\prime \prime
}>s_{1}^{\prime \prime }>p$.

Next, define the sequences $\{s_{i}^{\prime }\}_{i\geq 1}$ and $%
\{s_{i}^{\prime \prime }\}_{i\geq 1}$ iteratively as follows 
\begin{equation}
\frac{1}{s_{i}^{\prime }}=\frac{p}{s_{i-1}^{\prime \prime }}-\frac{2}{N}%
+\eta ,\quad \frac{1}{s_{i}^{\prime \prime }}=\frac{q}{s_{i-1}^{\prime }}-%
\frac{2}{N}+\eta ,\quad i\geq 3.  \label{lso20}
\end{equation}%
Then 
\begin{equation*}
\frac{1}{s_{i}^{\prime }}-\frac{1}{s_{i+1}^{\prime }}=p\Big(\frac{1}{%
s_{i-1}^{\prime \prime }}-\frac{1}{s_{i}^{\prime \prime }}\Big)=pq\Big(\frac{%
1}{s_{i-2}^{\prime }}-\frac{1}{s_{i-1}^{\prime }}\Big),
\end{equation*}%
and 
\begin{equation*}
\frac{1}{s_{i}^{\prime \prime }}-\frac{1}{s_{i+1}^{\prime \prime }}=q\Big(%
\frac{1}{s_{i-1}^{\prime }}-\frac{1}{s_{i}^{\prime }}\Big)=pq\Big(\frac{1}{%
s_{i-2}^{\prime \prime }}-\frac{1}{s_{i-1}^{\prime \prime }}\Big).
\end{equation*}%
Since $pq>1$, in view of \eqref{lso19}, we obtain 
\begin{equation}
\frac{1}{s_{i}^{\prime }}>\frac{1}{s_{i+1}^{\prime }},\quad \frac{1}{%
s_{i}^{\prime \prime }}>\frac{1}{s_{i+1}^{\prime \prime }},\quad i\geq 1
\label{lso21}
\end{equation}%
and 
\begin{equation}
\lim_{i\rightarrow +\infty }\Big(\frac{1}{s_{i}^{\prime }}-\frac{1}{%
s_{i+1}^{\prime }}\Big)=\lim_{i\rightarrow +\infty }(\frac{1}{s_{i}^{\prime
\prime }}-\frac{1}{s_{i+1}^{\prime \prime }})=+\infty .  \label{lso22}
\end{equation}%
Now, we ensure that there exists $i_{0}$ such that 
\begin{equation}
\frac{p}{s_{i_{0}}^{\prime \prime }}<\frac{2}{N}\text{ }\quad \text{or}\quad 
\frac{q}{s_{i_{0}}^{\prime }}<\frac{2}{N}.  \label{lso23}
\end{equation}%
In fact, if \eqref{lso23} is not true, that is $\frac{p}{s_{i}^{\prime
\prime }}\geq \frac{2}{N}$ and $\frac{q}{s_{i}^{\prime }}\geq \frac{2}{N}$
for all $i\geq 1$.

Then, by \eqref{lso20}, we see that $s_{i}^{\prime }>0$, $s_{i}^{\prime
\prime }>0$ for all $i\geq 1$ and hence, by \eqref{lso21}, 
\begin{equation*}
q<s_{1}^{\prime }<\dots <s_{i}^{\prime }<\dots ,\quad p<s_{1}^{\prime \prime
}<\dots <s_{i}^{\prime \prime }<\dots \,.
\end{equation*}%
Therefore%
\begin{equation*}
\left\vert \frac{1}{s_{i}^{\prime }}-\frac{1}{s_{i+1}^{\prime }}\right\vert
\leq \frac{1}{s_{i}^{\prime }}+\frac{1}{s_{i+1}^{\prime }}<\frac{2}{q}<2,%
\text{ for any }i\geq 1,
\end{equation*}%
which is a contradiction to \eqref{lso22}.

Let $i_0$ be the smallest number that satisfies \eqref{lso23}. We note that $%
i_0\geq 2$. Without loss of generality, we assume that 
\begin{equation}
\begin{gathered} \frac{p}{s_{i_0}''}<\frac{2}{N},\quad \frac{p}{s_i''}\geq
\frac{2}{N}\quad \text{for }1\leq i\leq i_0-1,\\ \frac{q}{s_i'}\geq
\frac{2}{N}\quad \text{for } 1\leq i\leq i_0. \end{gathered}  \label{lso24}
\end{equation}
Thus \eqref{lso20} yields 
\begin{equation*}
s_i^{\prime }>0\text{ for }1\leq i\leq i_0,\quad s_i^{\prime \prime }>0\text{
for }1\leq i\leq i_0+1,
\end{equation*}
which together with \eqref{lso21} leads to 
\begin{equation*}
q<\dots <s_{i_0-1}^{\prime }<s_{i_0}^{\prime }\quad p<\dots <s_{i_0}^{\prime
\prime }<s_{i_0+1}^{\prime \prime }.
\end{equation*}
Now, from \eqref{lso20}, for all $i\geq 2$ we have 
\begin{equation*}
\frac{N}{2}\Big( \frac{p}{s_{i-1}^{\prime \prime }}-\frac{1}{s_i^{\prime }}%
\Big) =1-\frac{N}{2}\eta =\frac{N}{2}\Big( \frac{q}{s_{i-1}^{\prime }}-\frac{%
1}{s_i^{\prime \prime }}\Big) .
\end{equation*}
Now, let us deal with the boundedness of $( u( t) ,v(t)) $ in $%
L^{s_i^{\prime }}( \mathbb{R}^N) \times L^{s_i^{\prime \prime }}( \mathbb{R}%
^N) $.

By using Lemmas \ref{LpqeP} and \ref{LpqeS}, it follows from \eqref{ms}
inductively that, for any $2\leq i\leq i_{0}$ and for any $t\in (0,T_{%
\mathrm{max}})$,%
\begin{eqnarray}
\Vert u(t)\Vert _{s_{i}^{\prime }} &\leq &\Vert P_{\gamma _{1}}(t)u_{0}\Vert
_{s_{i}^{\prime }}+C\int_{0}^{t}(t-\tau )^{\gamma _{1}-1-\frac{N\gamma _{1}}{%
2}(\frac{p}{s_{i-1}^{\prime \prime }}-\frac{1}{s_{i}^{\prime }})}\Vert
v(\tau )\Vert _{s_{i-1}^{\prime \prime }}^{p}d\tau  \notag \\
&\leq &C\Vert u_{0}\Vert _{s_{i}^{\prime }}+C\int_{0}^{t}(t-\tau )^{\gamma
_{1}-1-\gamma _{1}(1-\frac{N\eta }{2})}\Vert v(\tau )\Vert _{s_{i-1}^{\prime
\prime }}^{p}d\tau ,  \label{ls25}
\end{eqnarray}

and 
\begin{eqnarray}
\Vert v(t)\Vert _{s_{i}^{\prime \prime }} &\leq &\Vert P_{\gamma
_{2}}(t)v_{0}\Vert _{s_{i}^{\prime \prime }}+C\int_{0}^{t}(t-\tau )^{\gamma
_{2}-1+\frac{N\gamma _{2}}{2}(\frac{q}{s_{i-1}^{\prime }}-\frac{1}{%
s_{i}^{\prime \prime }})}\Vert u(\tau )\Vert _{s_{i-1}^{\prime }}^{q}d\tau 
\notag \\
&\leq &C\Vert v_{0}\Vert _{s_{i}^{\prime \prime }}+C\int_{0}^{t}(t-\tau
)^{\gamma _{2}-1-\gamma _{2}(1-\frac{N\eta }{2})}\Vert u(\tau )\Vert
_{s_{i-1}^{\prime }}^{q}d\tau ,  \label{lso26}
\end{eqnarray}%
for any $2\leq i\leq i_{0}+1$, for any $t\in (0,T_{\mathrm{max}})$.

From the H\"{o}lder inequality, we have%
\begin{eqnarray}
\Vert P_{\gamma _{1}}(t)u_{0}\Vert _{s_{i}^{\prime }} &\leq &\Vert
u_{0}\Vert _{s_{i}^{\prime }}\leq \Vert u_{0}\Vert _{s_{1}}^{\frac{s_{1}}{%
s_{i}^{\prime }}}\Vert u_{0}\Vert _{\infty }^{1-\frac{s_{1}}{s_{i}^{\prime }}%
}<\infty \text{,}  \notag \\
\Vert P_{\gamma _{2}}(t)v_{0}\Vert _{s_{i}^{\prime \prime }} &\leq &\Vert
v_{0}\Vert _{s_{i}^{\prime \prime }}\leq \Vert v_{0}\Vert _{s_{1}}^{\frac{%
s_{1}}{s_{i}^{\prime \prime }}}\Vert v_{0}\Vert _{\infty }^{1-\frac{s_{1}}{%
s_{i}^{\prime \prime }}}<\infty \text{, }  \label{interp} \\
\text{ since }u_{0} &\in &L^{s_{1}}\cap L^{\infty }\text{, }v_{0}\in
L^{s_{2}}\cap L^{\infty }.  \notag
\end{eqnarray}%
It follows then from \eqref{ls25}, \eqref{lso26} and \eqref{interp}\ that 
\begin{eqnarray}
u(t) &\in &L^{s_{i}^{\prime }}(\mathbb{R}^{N}),\quad \Vert u(t)\Vert
_{s_{i}^{\prime }}\leq C(1+t^{a_{i}}),\quad 1\leq \forall i\leq i_{0},\quad 
\notag \\
v(t) &\in &L^{s_{i}^{\prime \prime }}(\mathbb{R}^{N}),\quad \Vert v(t)\Vert
_{s_{i}^{\prime \prime }}\leq C(1+t^{b_{i}}),\quad 1\leq \forall i\leq
i_{0}+1,\;  \label{estlprim}
\end{eqnarray}%
for any $t\in (0,T_{\mathrm{max}})$ and for some positive constants $a_{i}$, 
$b_{i}$. Since $\frac{Np}{2s_{i_{0}^{\prime \prime }}}<1$, taking $%
s_{2}=s_{i_{0}}^{\prime \prime }$, \eqref{test} holds; hence $T_{\mathrm{max}%
}=+\infty $ and \eqref{test1} holds.

\noindent \textbf{Step 3:} First, we show the folowing decay estimates 
\begin{equation*}
\Vert u(t)\Vert _{s_{1}}\leq C(t+1)^{-\sigma _{1}},\quad \Vert v(t)\Vert
_{s_{2}}\leq C(t+1)^{-\sigma _{2}}\text{, for any }t\geq 0,
\end{equation*}%
where $s_{1}$ and $s_{2}$ are given by \eqref{S_1,S_2}.

According to Lemma \ref{poldec}, it suffices to prove that $\Vert u(t)\Vert
_{s_{1}}\leq C,\Vert v(t)\Vert _{s_{2}}\leq C$, for any $t\in \left[ 0,1%
\right] $.

To do this, we need to show 
\begin{equation}
\Vert u(t)\Vert _{\infty }\leq C,\quad \Vert v(t)\Vert _{\infty }\leq
C,\quad \text{for any }t\in \lbrack 0,1].  \label{lso10}
\end{equation}%
In fact, by applying Lemmas \ref{LpqeP} and \ref{LpqeS} to \eqref{ms} we see
that%
\begin{gather*}
\Vert u(s)\Vert _{\infty }\leq \Vert P_{\gamma _{1}}(s)u_{0}\Vert _{\infty
}+\int_{0}^{s}(s-\tau )^{\gamma _{1}-1}\Vert v(\tau )\Vert _{\infty
}^{p}d\tau \leq \Vert u_{0}\Vert _{\infty }+\int_{0}^{s}(s-\tau )^{\gamma
_{1}-1}\Vert v(\tau )\Vert _{\infty }^{p}d\tau , \\
\Vert v(s)\Vert _{\infty }\leq \Vert P_{\gamma _{2}}(s)v_{0}\Vert _{\infty
}+\int_{0}^{s}(s-\tau )^{\gamma _{2}-1}\Vert u(\tau )\Vert _{\infty
}^{q}d\tau \leq \Vert v_{0}\Vert _{\infty }+\int_{0}^{s}(s-\tau )^{\gamma
_{2}-1}\Vert u(\tau )\Vert _{\infty }^{q}d\tau ,
\end{gather*}%
for $0\leq s\leq t$. The last two inequalities give, for $0\leq s\leq t\leq
1 $ 
\begin{gather*}
\sup_{0\leq s\leq t}\Vert u(s)\Vert _{\infty }\leq \Vert u_{0}\Vert _{\infty
}+\frac{1}{\gamma _{1}}\left( \sup_{0\leq \tau \leq t}\Vert v(\tau )\Vert
_{\infty }\right) ^{p}t^{\gamma _{1}}\leq \Vert u_{0}\Vert _{\infty }+\frac{1%
}{\gamma _{1}}\left( \sup_{0\leq \tau \leq t}\Vert v(\tau )\Vert _{\infty
}\right) ^{p}, \\
\sup_{0\leq s\leq t}\Vert v(s)\Vert _{\infty }\leq \Vert v_{0}\Vert _{\infty
}+\frac{1}{\gamma _{2}}\left( \sup_{0\leq \tau \leq t}\Vert u(\tau )\Vert
_{\infty }\right) ^{q}t^{\gamma _{2}}\leq \Vert v_{0}\Vert _{\infty }+\frac{1%
}{\gamma _{2}}\left( \sup_{0\leq \tau \leq t}\Vert u(\tau )\Vert _{\infty
}\right) ^{q}.
\end{gather*}%
Using the second inequality into first inequality, it yields that%
\begin{eqnarray*}
\sup_{0\leq s\leq t}\Vert u(s)\Vert _{\infty } &\leq &\Vert u_{0}\Vert
_{\infty }+\frac{1}{\gamma _{1}}\left( \Vert v_{0}\Vert _{\infty }+\frac{1}{%
\gamma _{2}}\left( \sup_{0\leq \tau \leq t}\Vert u(\tau )\Vert _{\infty
}\right) ^{q}\right) ^{p} \\
&\leq &C\left( \Vert u_{0}\Vert _{\infty }+\Vert v_{0}\Vert _{\infty
}^{p}+\left( \sup_{0\leq \tau \leq t}\Vert u(\tau )\Vert _{\infty }\right)
^{pq}\right)
\end{eqnarray*}%
So, arguing as in the first step by setting $h\left( t\right)
=\sup\limits_{0\leq s\leq t}\Vert u(s)\Vert _{\infty }$ and $A=\Vert
u_{0}\Vert _{\infty }+\Vert v_{0}\Vert _{\infty }^{p}$, we obtain%
\begin{equation*}
h\left( t\right) \leq A+Ch^{pq}\left( t\right) \text{, for any }t\leq 1\text{%
,}
\end{equation*}%
which implies \eqref{lso10} for $A$ small since $pq>1$.

We see from \eqref{lso1} and Lemmas \ref{LpqeP}, \ref{LpqeS} that

\begin{equation*}
\Vert u(t)\Vert _{s_{1}}\leq C\Vert u_{0}\Vert _{s_{1}}+C\int_{0}^{t}(t-\tau
)^{\gamma _{1}-1}\left\Vert \left\vert v\left( \tau \right) \right\vert
^{p}\right\Vert _{s_{1}}d\tau ,
\end{equation*}

where $s_{1}$ given explicitly by \eqref{S_1,S_2}. Therefore%
\begin{equation*}
\Vert u(t)\Vert _{s_{1}}=C\Vert u_{0}\Vert _{s_{1}}+C\int_{0}^{t}(t-\tau
)^{\gamma _{1}-1}\left\Vert v\left( \tau \right) \right\Vert
_{ps_{1}}^{p}d\tau .
\end{equation*}%
By interpolation inequality $\left\Vert v\left( \tau \right) \right\Vert
_{ps_{1}}^{p}\leq \left\Vert v\left( \tau \right) \right\Vert _{s_{2}}^{%
\frac{s_{2}}{s_{1}}}\left\Vert v\left( \tau \right) \right\Vert _{\infty
}^{p\left( 1-\frac{s_{2}}{ps_{1}}\right) }$, we get%
\begin{equation}
\Vert u(t)\Vert _{s_{1}}=C\Vert u_{0}\Vert _{s_{1}}+C\sup_{\tau \in \left(
0,t\right) }\left\Vert v\left( \tau \right) \right\Vert _{\infty }^{p\left(
1-\frac{s_{2}}{ps_{1}}\right) }\int_{0}^{t}(t-\tau )^{\gamma
_{1}-1}\left\Vert v\left( \tau \right) \right\Vert _{s_{2}}^{\frac{s_{2}}{%
s_{1}}}d\tau .  \label{lso11}
\end{equation}

Now, using \eqref{lso9} and \eqref{lso10} in \eqref{lso11}, we obtain%
\begin{equation*}
\Vert u(t)\Vert _{s_{1}}\leq C\Vert u_{0}\Vert _{s_{1}}+C\int_{0}^{t}(t-\tau
)^{\gamma _{1}-1}\tau ^{-\frac{s_{2}}{s_{1}}\sigma _{2}}d\tau ,
\end{equation*}%
provided that $\frac{s_{2}}{s_{1}}\sigma _{2}<1$. On the other hand, since $%
s_{1}$and $s_{2}$ satisfy 
\begin{equation*}
\frac{s_{1}}{s_{2}}\sigma _{1}=\frac{(1-\delta )(p\gamma _{2}+\gamma
_{1})s_{1}}{(pq-1)s_{2}}\leq \gamma _{2},\quad \frac{s_{2}}{s_{1}}\sigma
_{2}=\frac{(1-\delta )(q\gamma _{1}+\gamma _{2})s_{2}}{(pq-1)s_{1}}\leq
\gamma _{1},
\end{equation*}%
we get $\gamma _{1}-\frac{s_{2}}{s_{1}}\sigma _{2}\geq 0$ and consequently%
\begin{equation*}
\Vert u(t)\Vert _{s_{1}}\leq C\Vert u_{0}\Vert _{s_{1}}+Ct^{\gamma _{1}-%
\frac{s_{2}}{s_{1}}\sigma _{2}}\leq C,\quad \text{for all }t\in \lbrack 0,1].
\end{equation*}

Analogously, 
\begin{equation*}
\Vert v(t)\Vert _{s_{2}}\leq C\quad \text{for all }t\in \lbrack 0,1].
\end{equation*}%
From \eqref{lso8}, \eqref{lso9}, \eqref{lso10} and Lemma \ref{poldec}, we
conclude that 
\begin{equation}
\Vert u(t)\Vert _{s_{1}}\leq C(t+1)^{-\frac{(1-\delta )(\gamma _{1}+p\gamma
_{2})}{pq-1}},\quad \Vert v(t)\Vert _{s_{2}}\leq C(t+1)^{-\frac{(1-\delta
)(\gamma _{2}+q\gamma _{1})}{pq-1}},  \label{lso13}
\end{equation}%
for any $t\geq 0$.

Next, we derive $L^{\infty }$-decay estimates. Let 
\begin{equation*}
\sigma _{1}=\frac{(1-\delta )(p\gamma _{2}+\gamma _{1})}{(pq-1)},\quad
\sigma _{2}=\frac{(1-\delta )(q\gamma _{1}+\gamma _{2})}{(pq-1)}.
\end{equation*}%
If $\frac{pN}{2s_{2}}<1$, by taking $\xi =s_{2}$ in \eqref{lso16} and using %
\eqref{lso13}, we obtain 
\begin{equation}
\Vert u(t)\Vert _{\infty }\leq Ct^{-\frac{N}{2}\gamma _{1}}\Vert u_{0}\Vert
_{1}+C\int_{0}^{t}(t-\tau )^{\gamma _{1}-1-\frac{N\gamma _{1}}{2}\frac{p}{%
s_{2}}}\tau ^{-p\sigma _{2}}d\tau  \label{lso27}
\end{equation}%
and 
\begin{equation*}
p\sigma _{2}<1,\quad \gamma _{1}-\frac{N\gamma _{1}}{2}\frac{p}{s_{2}}%
-p\sigma _{2}=-\frac{[\gamma _{1}+\gamma _{1}p\delta +(1-\delta )p\gamma
_{2}]}{pq-1}.
\end{equation*}%
On the other hand, we have 
\begin{equation}
\frac{\gamma _{1}+\gamma _{1}p\delta +p\gamma _{2}(1-\delta )}{pq-1}<\frac{%
\gamma _{1}+p\gamma _{2}}{pq-1}\leq \frac{N}{2}\gamma _{1}.  \label{comp}
\end{equation}%
Then, it follows from \eqref{lso27}, \eqref{comp} and lemma \ref{poldec}
that 
\begin{equation*}
\Vert u(t)\Vert _{\infty }\leq Ct^{-\frac{N}{2}\gamma _{1}}+Ct^{-\frac{%
[\gamma _{1}+\gamma _{1}p\delta +(1-\delta )p\gamma _{2}]}{pq-1}}.
\end{equation*}%
Thus%
\begin{equation}
\Vert u(t)\Vert _{\infty }\leq C(1+t)^{-\frac{[\gamma _{1}+\gamma
_{1}p\delta +(1-\delta )p\gamma _{2}]}{pq-1}}.  \label{lso28}
\end{equation}%
for any $t\geq 0$.

Similarly, if $\frac{qN}{2s_{1}}<1$, then one can find that 
\begin{equation}
\Vert v(t)\Vert _{\infty }\leq C(1+t)^{-\frac{[\gamma _{2}+\gamma
_{2}q\delta +(1-\delta )q\gamma _{1}]}{pq-1}},\quad \text{for }t\geq 0.
\label{lso29}
\end{equation}%
At the same time, \eqref{lso28} holds as $pN/(2s_{2})\leq qN/(2s_{1})$.

In particular, if $pq>q+2$, and $\gamma _{1}q^{2}>2q+1,$ we can choose%
\newline
\begin{equation*}
\delta >\max \left\{ 1-\frac{(pq-1)}{\gamma _{2}q(p+1)},1-\frac{\gamma
_{1}\left( pq-1\right) }{\gamma _{2}(p+1)}\right\}
\end{equation*}%
and $\delta \approx \max \left\{ 1-\frac{(pq-1)}{\gamma _{2}q(p+1)},1-\frac{%
\gamma _{1}\left( pq-1\right) }{\gamma _{2}(p+1)}\right\} $ such that $%
qN/(2s_{1})<1$. \newline
Therefore, estimates \eqref{lso28} and \eqref{lso29} hold.

It is useful to note that $N\leq 2$ implies $pN/(2s_{2})<1$ and $%
qN/(2s_{1})<1$ implies $pq>2+q$. \newline
It remains to consider the following two cases:

(1) The case $N>2$, $\frac{Np}{2s_{2}}<1$ and $\frac{Nq}{2s_{1}}\geq 1$. Let 
\begin{equation*}
\sigma ^{\prime }=\frac{\gamma _{1}+\gamma _{1}p\delta +(1-\delta )p\gamma
_{2}}{pq-1}.
\end{equation*}%
For a positive $\mu $ such that $\mu <\min \{\sigma ^{\prime },\sigma _{1}\}$
and $q\mu <1$, since $N>2$ and $q>1$, we can choose $k>0$ such that $k>\frac{%
qN}{2}$ and $q\mu +\frac{qN\gamma _{2}}{2k}>\gamma _{2}$. Since $s_{1}\leq
qN/2$, we have $k>s_{1}$.

Using the interpolation inequality, 
\begin{equation*}
\Vert u(t)\Vert _{k}\leq \Vert u(t)\Vert _{\infty }^{(k-s_{1})/k}\Vert
u(t)\Vert _{s_{1}}^{s_{1}/k}\leq Ct^{-\sigma ^{\prime
}(k-s_{1})/k}t^{-\sigma _{1}s_{1}/k}\quad \text{for all }t>0,
\end{equation*}%
it follows from \eqref{lso8} and \eqref{lso28} that 
\begin{equation*}
\Vert u(t)\Vert _{k}\leq Ct^{-\mu }\quad \text{for all }t>0.
\end{equation*}%
Whereupon,%
\begin{eqnarray}
\Vert v(t)\Vert _{\infty } &\leq &\Vert P_{\gamma _{2}}(t)v_{0}\Vert
_{\infty }+C\int_{0}^{t}(t-\tau )^{\gamma _{2}-1-\frac{Nq}{2k}\gamma
_{2}}\Vert u(\tau )\Vert _{k}^{q}d\tau  \notag \\
&\leq &Ct^{-\frac{N}{2}\gamma _{2}}\Vert v_{0}\Vert
_{1}+C\int_{0}^{t}(t-\tau )^{\gamma _{2}-1-\frac{N\gamma _{2}q}{2k}}\tau
^{-q\mu }d\tau  \notag \\
&\leq &C(t^{-\frac{N}{2}\gamma _{2}}+t^{\gamma _{2}-\frac{N\gamma _{2}q}{2k}%
-q\mu })  \notag \\
&\leq &Ct^{-\alpha },  \label{lso30}
\end{eqnarray}

for any $t>0$, where $\alpha =\min \{\frac{N}{2}\gamma _{2},-\gamma _{2}+%
\frac{N\gamma _{2}q}{2k}+q\mu \}>0$.\newline

From \eqref{lso10} and \eqref{lso30}, we infer that 
\begin{equation*}
\Vert v(t)\Vert _{\infty }\leq C(1+t)^{-\alpha }\quad \text{for all }t\geq 0.
\end{equation*}%
We remark that, in the particular case $p=1,$ $q>3$ and $q^{2}>\max \left\{
4\gamma _{2}q+1,\frac{4\gamma _{2}+\gamma _{1}}{\gamma _{1}}\right\} $, we
can choose 
\begin{eqnarray*}
\delta &>&\max \left\{ 1-\frac{(pq-1)}{\gamma _{2}q(p+1)},1-\frac{\gamma
_{1}\left( pq-1\right) }{\gamma _{2}(p+1)},1-\frac{\left( pq-1\right) }{q+1}%
\right\} \\
&=&\max \left\{ 1-\frac{(q-1)}{2\gamma _{2}q},1-\frac{\gamma _{1}\left(
q-1\right) }{2\gamma _{2}},1-\frac{\left( q-1\right) }{q+1}\right\}
\end{eqnarray*}%
and $\delta \approx \max \left\{ 1-\frac{(q-1)}{2\gamma _{2}q},1-\frac{%
\gamma _{1}\left( q-1\right) }{2\gamma _{2}},1-\frac{\left( q-1\right) }{q+1}%
\right\} $ such that $N/(2s_{2})<1$. Therefore, we have the estimate %
\eqref{lso28}.\newline

(2) The case: $N>2$, $qN/(2s_{1})\geq 1$, $pN/(2s_{2})\geq 1$, $q\geq p>1,$
and $\gamma _{1}\leq \gamma _{2}$. It needs a careful handling and we need
to restrict further the choice of $\delta $. \newline
From $\max \left\{ \frac{q+1}{pq(p+1)},\frac{pq-1}{pq(p+1)},\gamma _{2}/p,%
\sqrt{\frac{\gamma _{2}}{pq}}\right\} <\gamma _{1}\leq \gamma _{2}<1$ and $\
pq>1$, we get 
\begin{equation*}
\max \left\{ 1-\frac{(pq-1)}{\gamma _{2}q(p+1)},1-\frac{\gamma _{1}\left(
pq-1\right) }{\gamma _{2}(p+1)},1-\frac{\left( pq-1\right) }{q+1}\right\}
<\min \left\{ 1-\frac{(pq-1)}{\gamma _{1}pq(p+1)},\frac{N(pq-1)}{2q(p+1)}%
\right\} .
\end{equation*}%
So, we select $\delta >0$ such that 
\begin{equation*}
\max \left\{ 1-\frac{(pq-1)}{\gamma _{2}q(p+1)},1-\frac{\gamma _{1}\left(
pq-1\right) }{\gamma _{2}(p+1)},1-\frac{\left( pq-1\right) }{q+1}\right\}
<\delta <\min \left\{ \frac{N(pq-1)}{2q(p+1)},1-\frac{(pq-1)}{\gamma
_{1}pq(p+1)}\right\} .
\end{equation*}%
We get immediately $p\sigma _{2}>1/q$ and $q\sigma _{1}>1/p$. Further, we
notice that there exist $\varepsilon \in (0,1)$ and $\beta <1$ close to $1$
such that 
\begin{equation}
p\sigma _{2}-\varepsilon >1/q,\quad q\sigma _{1}-\varepsilon >1/p,\quad
1/p<\beta -\varepsilon ,\quad 1/q<\beta -\varepsilon .  \label{epsilonbeta}
\end{equation}

By taking $\eta =2\varepsilon (1-\delta )/N$, we find the integer $i_{0}$ as
in the step 2, and, without loss of generality, we assume that \eqref{lso24}
holds. We choose $\beta $ in addition to \eqref{epsilonbeta} satisfying 
\begin{equation}
\gamma _{1}<\gamma _{1}\frac{pN}{2s_{i_{0}}^{\prime \prime }}+\gamma
_{2}\beta ,\quad \text{since }\gamma _{2}\geq \gamma _{1}.  \label{betgamma}
\end{equation}%
As 
\begin{equation*}
\delta <\frac{N(pq-1)}{2(p+1)q}\leq \frac{N(pq-1)}{2(q+1)p},
\end{equation*}%
and $\beta <1$, we have 
\begin{equation}
\beta +\frac{(p+1)q\delta }{(pq-1)}<1+\frac{N}{2},\quad \beta +\frac{%
(q+1)p\delta }{(pq-1)}<1+\frac{N}{2}.  \label{betadelt}
\end{equation}%
For $2\leq i\leq i_{0}-1$, define $r_{i+1}^{\prime }$ and $r_{i+1}^{\prime
\prime }$ inductively as follows: 
\begin{gather*}
\frac{1}{r_{2}^{\prime }}=\frac{1}{s_{2}^{\prime }}+\frac{2}{N}[p\sigma
_{2}-\varepsilon (1-\delta )],\quad \frac{1}{r_{2}^{\prime \prime }}=\frac{1%
}{s_{2}^{\prime \prime }}+\frac{2}{N}[q\sigma _{1}-\varepsilon (1-\delta )],
\\
\frac{1}{r_{i+1}^{\prime }}=\frac{1}{s_{i+1}^{\prime }}+\frac{2}{N}[\beta
-\varepsilon (1-\delta )],\quad \frac{1}{r_{i+1}^{\prime \prime }}=\frac{1}{%
s_{i+1}^{\prime \prime }}+\frac{2}{N}[\beta -\varepsilon (1-\delta )].
\end{gather*}%
It is clear that $r_{i}^{\prime }$, $r_{i}^{\prime \prime }>0$ and $%
r_{i}^{\prime }<s_{i}^{\prime }$, $r_{i}^{\prime \prime }<s_{i}^{\prime
\prime }$ for all $2\leq i\leq i_{0}$. A simple calculation shows that $%
r_{i}^{\prime }$, $r_{i}^{\prime \prime }>1$. As $s_{i}^{\prime }$ and $%
s_{i}^{\prime \prime }$ are increasing in $i$ for $1\leq i\leq i_{0}$; we
have 
\begin{align*}
\frac{1}{r_{i+1}^{\prime }}& <\frac{1}{s_{2}^{\prime }}+\frac{2}{N}[\beta
-\varepsilon (1-\delta )] \\
& =\frac{p}{s_{1}^{\prime \prime }}-\frac{2}{N}+\frac{2}{N}\varepsilon
(1-\delta )+\frac{2}{N}[\beta -\varepsilon (1-\delta )] \\
& =\frac{2}{N}\Big(\frac{p(q+1)\delta }{pq-1}+\beta -1\Big)<1,
\end{align*}%
from \eqref{betadelt}; therefore $r_{i+1}^{\prime }>1$. Similarly, we can
check that $r_{i+1}^{\prime \prime }>1$.

From \eqref{test1} and \eqref{estlprim}, we see that there exists a positive
constant $C$ such that 
\begin{equation}
\| u(t)\| _{\infty },\; \| v(t)\|_{\infty },\; \| u(t)\| _{k_1},\; \| v(t)\|
_{k_2} \leq C\quad \text{for }0\leq t\leq 1,  \label{luvinfty}
\end{equation}
for any $s_1^{\prime }\leq k_1\leq s_{i_0}^{\prime }$, $s_1^{\prime \prime
}\leq k_2\leq s_{i_0}^{\prime \prime }$.

Furthermore, since $1-\eta N/2=1-\varepsilon (1-\delta )$ and $p\sigma
_{2}<1 $, using \eqref{ls25}, \eqref{lso26} with the help of \eqref{lso8}
and \eqref{lso9}, we arrive at the estimate 
\begin{align*}
\Vert u(t)\Vert _{s_{2}^{\prime }}& \leq \Vert P_{\gamma _{1}}(t)u_{0}\Vert
_{s_{2}^{\prime }}+C\int_{0}^{t}(t-\tau )^{\gamma _{1}-1-\gamma
_{1}(1-\varepsilon (1-\delta ))}\Vert u(\tau )\Vert _{s_{2}^{\prime \prime
}}^{p}d\tau \\
& \leq Ct^{-\frac{N}{2}\gamma _{1}(\frac{1}{r_{2}^{\prime }}-\frac{1}{%
s_{2}^{\prime }})}\Vert u_{0}\Vert _{s_{2}^{\prime }}+C\int_{0}^{t}(t-\tau
)^{\gamma _{1}-1-\gamma _{1}(1-\varepsilon (1-\delta ))}\tau ^{-p\sigma
_{2}}d\tau \\
& \leq Ct^{-\gamma _{1}(p\sigma _{2}-\varepsilon (1-\delta ))}\Vert
u_{0}\Vert _{s_{2}^{\prime }}+C\int_{0}^{t}(t-\tau )^{\gamma _{1}-1-\gamma
_{1}(1-\varepsilon (1-\delta ))}\tau ^{-p\sigma _{2}}d\tau \\
& \leq Ct^{-\gamma _{1}(p\sigma _{2}-\varepsilon (1-\delta ))}\quad \text{%
for all }t>0.
\end{align*}%
Similarly, 
\begin{equation*}
\Vert v(t)\Vert _{s_{2}^{\prime \prime }}\leq Ct^{-\gamma _{2}(p\sigma
_{1}-\varepsilon (1-\delta ))}\quad \text{for all }t>0.
\end{equation*}%
In view of \eqref{epsilonbeta} and $\beta <1$, we conclude, thanks to lemma %
\ref{poldec}, that 
\begin{equation*}
\Vert u(t)\Vert _{s_{2}^{\prime }}\leq Ct^{-\gamma _{1}\beta /q},\quad \Vert
v(t)\Vert _{s_{2}^{\prime \prime }}\leq Ct^{-\gamma _{2}\beta /p}\quad \text{%
for all }t>0.
\end{equation*}%
An iterative argument gives 
\begin{gather*}
\Vert u(t)\Vert _{s_{i_{0}}^{\prime }}\leq Ct^{-\gamma _{1}(\beta
-\varepsilon (1-\delta ))}\leq Ct^{-\gamma _{1}\beta /q}\quad \text{for all }%
t>0, \\
\Vert v(t)\Vert _{s_{i_{0}}^{\prime \prime }}\leq Ct^{-\gamma _{2}(\beta
-\varepsilon (1-\delta ))}\leq Ct^{-\gamma _{2}\beta /p}\quad \text{for all }%
t>0.
\end{gather*}%
Therefore, by \eqref{lso15} and \eqref{lso16}, we have 
\begin{align*}
\Vert u(t)\Vert _{\infty }& \leq Ct^{-\frac{N}{2}\gamma _{1}}\Vert
u_{0}\Vert _{1}+C\int_{0}^{t}(t-\tau )^{\gamma _{1}-1-\gamma _{1}\frac{pN}{%
2s_{i_{0}}^{\prime \prime }}}\Vert v(\tau )\Vert _{s_{i_{0}}^{\prime \prime
}}^{p}d\tau \\
& \leq Ct^{-\frac{N}{2}\gamma _{1}}\Vert u_{0}\Vert
_{1}+C\int_{0}^{t}(t-\tau )^{\gamma _{1}-1-\gamma _{1}\frac{pN}{%
2s_{i_{0}}^{\prime \prime }}}\tau ^{-\gamma _{2}\beta }d\tau \\
& \leq C(t^{-\frac{N}{2}\gamma _{1}}+t^{\gamma _{1}-\gamma _{1}\frac{pN}{%
2s_{i_{0}}^{\prime \prime }}-\gamma _{2}\beta })\leq Ct^{-\tilde{\sigma}},
\end{align*}%
where $\sigma ^{\prime \prime }=\min \{\frac{N}{2}\gamma _{1},\gamma _{1}%
\frac{pN}{2s_{i_{0}}^{\prime \prime }}-\gamma _{1}+\gamma _{2}\beta \}>0$
from \eqref{betgamma}.

Since $\frac{Nq}{2s_1}\geq 1$, using similar arguments as for the case $%
\frac{Np}{2s_2}<1$ and $\frac{Nq}{2s_1}\geq 1$, we obtain $\| v(t) \|
_{\infty }\leq Ct^{-\sigma _1^{\prime \prime }}$ for some $\sigma _1^{\prime
\prime }>0$ and for every $t>0$. This completes the proof of the Theorem.
\end{proof}

\begin{remark}
\textrm{In the particular case }$N>2$\textrm{, }$qN/(2s_{1})\geq 1$\textrm{, 
}$pN/(2s_{2})\geq 1$\textrm{, }$q>p=1$\textrm{\ and }$q^{2}\leq 4\gamma
_{1}q+1$\textrm{, using the above method, we obtain 
\begin{equation*}
\Vert u(t)\Vert _{\infty }\leq Ct^{-\sigma ^{\prime \prime }},\quad t>0,
\end{equation*}%
where $\sigma ^{\prime \prime }=\min \{\frac{N}{2}\gamma _{1},\frac{pN}{%
2s_{i_{0}}^{\prime \prime }}\gamma _{1}-\gamma _{1}+\gamma _{2}(\beta
-\varepsilon (1-\delta ))p\}$. Here, $\varepsilon >0$ can be arbitrarily
small, and $\beta $ can be arbitrarily close to $1$. However, since $%
s_{i_{0}}^{\prime \prime }$ depends on $\varepsilon $ and $s_{i_{0}}^{\prime
\prime }$ is decreasing in $\varepsilon $, it is not clear that $\sigma
^{\prime \prime }$ is positive.}
\end{remark}

\begin{proof}[Proof of Theorem \protect\ref{NEG}]
\quad

\noindent\textbf{Case: $p>1, q>1$.} The proof proceeds by contradiction.
Suppose that $(u,v)$ is a nontrivial solution of \eqref{sys1} which exists
globally in time. We make the judicious choice 
\begin{equation*}
\varphi ( t,x) =\varphi _1( t) \varphi _2( x) =\varphi _1\big( \frac{t}{%
T^{\lambda }}\big) \Phi ^{l}\big( \frac{| x| }{T^{2}}\big) ,
\end{equation*}
where $\Phi \in C_0^{\infty }( \mathbb{R}) $, $0\leq \Phi(z) \leq 1$ is such
that 
\begin{equation*}
\Phi ( z) =%
\begin{cases}
1 & \text{if }| z| \leq 1, \\ 
0 & \text{if }| z| >2,%
\end{cases}%
\end{equation*}
and 
\begin{equation*}
\varphi _1( t) =%
\begin{cases}
( 1-\frac{t}{T^{\lambda }}) ^{l} & \text{if }t\leq T^{\lambda }, \\ 
0 & \text{if }t>T^{\lambda },%
\end{cases}%
\end{equation*}
where $l>\max \{ 1,\frac{q}{q-1}\gamma _1-1,\frac{p}{p-1}\gamma_2-1\} $. We
denote by $Q_{T^{\lambda }}:=\mathbb{R}^N\times [ 0,T^{\lambda }]$.

From Definition \ref{Weaks}, of the weak solution, we have

\begin{eqnarray}
&&\int_{Q_{T^{\lambda }}}|v|^{p}\varphi _{2}(x)\varphi
_{1}(t)dx\,dt+T^{\lambda (1-\gamma _{1})}\int_{\mathbb{R}^{N}}u_{0}(x)%
\varphi _{2}(x)dx  \notag \\
&=&\int_{Q_{T^{\lambda }}}\varphi _{2}(x)uD_{t|T^{\lambda }}^{\gamma
_{1}}\varphi _{1}(t)dx\,dt-\int_{Q_{T^{\lambda }}}\Delta \varphi
_{2}(x)\varphi _{1}(t)u\,dx\,dt,  \label{formu1}
\end{eqnarray}%
\newline
\begin{eqnarray}
&&\int_{Q_{T^{\lambda }}}|u|^{q}\varphi _{2}(x)\varphi
_{1}(t)dx\,dt+T^{\lambda (1-\gamma _{2})}\int_{\mathbb{R}^{N}}v_{0}\varphi
_{2}(x)dx  \notag \\
&=&\int_{Q_{T^{\lambda }}}\varphi _{2}(x)uD_{t|T^{\lambda }}^{\gamma
_{2}}\varphi _{1}(t)dx\,dt-\int_{Q_{T^{\lambda }}}\varphi _{1}(t)\Delta
\varphi _{2}(x)u\,dx\,dt.  \label{formu2}
\end{eqnarray}%
Using H\"{o}lder's inequality with exponents $q$ and $q^{\prime }$ ($%
q+q^{\prime }=qq^{\prime }$), to the right-hand sides of \eqref{formu1} and %
\eqref{formu2}, we obtain 
\begin{eqnarray*}
&&\displaystyle\int_{Q_{T^{\lambda }}}u\varphi _{2}(x)D_{t|T^{\lambda
}}^{\gamma _{1}}\varphi _{1}(t)dx\,dt \\
&=&\displaystyle\int_{Q_{T^{\lambda }}}u|\varphi _{1}(t)|^{1/q}|\varphi
_{2}(x)|^{1-\frac{1}{q}+\frac{1}{q}}|\varphi _{1}(t)|^{-1/q}D_{t|T^{\lambda
}}^{\gamma _{1}}\varphi _{1}(t)dx\,dt \\
&\leq &\Big(\displaystyle\int_{Q_{T^{\lambda _{1}}}}|D_{t|T^{\lambda
}}^{\gamma _{1}}\varphi _{1}(t)|^{q^{\prime }}|\varphi _{1}(t)|^{-q^{\prime
}/q}|\varphi _{2}(x)|^{(1-\frac{1}{q})q^{\prime }}dx\,dt\Big)^{1/q^{\prime }}
\\
&&\quad \times \Big(\displaystyle\int_{Q_{T^{\lambda }}}|u|^{q}\varphi
_{1}\varphi _{2}dx\,dt\Big)^{1/q},
\end{eqnarray*}%
and 
\begin{eqnarray*}
&&\displaystyle\int_{Q_{T^{\lambda }}}u|\Delta \varphi _{2}(x)|\varphi
_{1}(t)dx\,dt \\
&\leq &\Big(\int_{\mathbb{R}^{N}}|\Delta \varphi _{2}(x)|^{q^{\prime
}}|\varphi _{2}(x)|^{-q^{\prime }/q}dx\displaystyle\int_{0}^{T^{\lambda
}}|\varphi _{1}(t)|^{(1-\frac{1}{q})q^{\prime }}dt\Big)^{1/q^{\prime }} \\
&&\quad \times \Big(\displaystyle\int_{Q_{T^{\lambda }}}|u|^{q}\varphi
_{1}\varphi _{2}dx\,dt\Big)^{1/q}.
\end{eqnarray*}%
Setting 
\begin{eqnarray*}
\mathcal{A}(\sigma ,\kappa ,\kappa ^{\prime }) &=&\Big(\displaystyle%
\int_{Q_{T^{\lambda _{1}}}}|D_{t|T^{\lambda }}^{\sigma }\varphi
_{1}(t)|^{\kappa ^{\prime }}|\varphi _{1}(t)|^{-\frac{\kappa ^{\prime }}{%
\kappa }}|\varphi _{2}(x)|^{(1-\frac{1}{\kappa })\kappa ^{\prime }}dx\,dt%
\Big)^{1/\kappa ^{\prime }}, \\
\mathcal{B}(\kappa ,\kappa ^{\prime }) &=&\Big(\displaystyle%
\int_{Q_{T^{\lambda _{1}}}}|\Delta \varphi _{2}(x)|^{\kappa ^{\prime
}}|\varphi _{2}(x)|^{-\frac{\kappa ^{\prime }}{\kappa }}|\varphi
_{1}(t)|^{(1-\frac{1}{\kappa })\kappa ^{\prime }}dx\,dt\Big)^{1/\kappa
^{\prime }},
\end{eqnarray*}%
and gathering the above estimates, we obtain

\begin{eqnarray}
&&\displaystyle\int_{Q_{T^{\lambda }}}|v|^{p}\varphi _{1}(t)\varphi
_{2}(x)dx\,dt+T^{\lambda (1-\gamma _{1})}\displaystyle\int_{\mathbb{R}%
^{N}}u_{0}\varphi _{2}(x)dx  \notag \\
&\leq &\mathcal{A}(\gamma _{1},q,q^{\prime })\Big(\displaystyle%
\int_{Q_{T^{\lambda }}}|u|^{q}\varphi _{1}(t)\varphi _{2}dx\,dt\Big)^{1/q} 
\notag \\
&&\quad +\mathcal{B}(q,q^{\prime })\Big(\displaystyle\int_{Q_{T^{\lambda
}}}|u|^{q}\varphi _{1}(t)\varphi _{2}dx\,dt\Big)^{1/q}.  \label{in1}
\end{eqnarray}%
Similarly, we obtain

\begin{eqnarray}
&&\int_{Q_{T^{\lambda }}}|u|^{q}\varphi _{2}(x)\varphi
_{1}(t)dtdx+T^{\lambda (1-\gamma _{2})}\int_{\mathbb{R}^{N}}v_{0}\varphi
_{2}(x)dx  \notag \\
&\leq &\mathcal{A}(\gamma _{2},p,p^{\prime })\Big(\displaystyle%
\int_{Q_{T^{\lambda }}}|v|^{p}\varphi _{1}\varphi _{2}dx\,dt\Big)^{1/p}+%
\mathcal{B}(p,p^{\prime })\Big(\displaystyle\int_{Q_{T^{\lambda
}}}|v|^{p}\varphi _{1}\varphi _{2}dx\,dt\Big)^{1/p}  \label{in2}
\end{eqnarray}

Consequently, 
\begin{eqnarray*}
&&\displaystyle\int_{Q_{T^{\lambda }}}|v|^{p}\varphi _{1}(t)\varphi
_{2}(x)dx\,dt+CT^{\lambda (1-\gamma _{1})}\displaystyle\int_{\mathbb{R}%
^{N}}u_{0}\varphi _{2}(x)dx \\
&\leq &\mathcal{A}\Big(\displaystyle\int_{Q_{T^{\lambda }}}|u|^{q}\varphi
_{1}\varphi _{2}dx\,dt\Big)^{1/q},
\end{eqnarray*}%
and 
\begin{eqnarray*}
&&\displaystyle\int_{Q_{T^{\lambda }}}|u|^{q}\varphi _{1}(t)\varphi
_{2}(x)dx\,dt+CT^{\lambda (1-\gamma _{2})}\displaystyle\int_{\mathbb{R}%
^{N}}v_{0}\varphi _{2}(x)dx \\
&\leq &\mathcal{B}\Big(\displaystyle\int_{Q_{T^{\lambda }}}|v|^{p}\varphi
_{1}\varphi _{2}dx\,dt\Big)^{1/p},
\end{eqnarray*}%
where 
\begin{equation*}
\mathcal{A}=\mathcal{A}(\gamma _{1},q,q^{\prime })+\mathcal{B}(q,q^{\prime
}),\quad \mathcal{B}=\mathcal{A}(\gamma _{2},p,p^{\prime })+\mathcal{B}%
(p,p^{\prime }).
\end{equation*}%
Using inequalities \eqref{in1} and \eqref{in2} in the last two inequalities,
we obtain

\begin{eqnarray*}
&&\displaystyle\int_{Q_{T^{\lambda }}}|v|^{p}\varphi _{1}(t)\varphi
_{2}(x)\,dx\,dt+CT^{\lambda (1-\gamma _{1})}\displaystyle\int_{\mathbb{R}%
^{N}}u_{0}\varphi _{2}(x)dx \\
&\leq &\mathcal{A}\mathcal{B}^{1/q}\Big(\displaystyle\int_{Q_{T^{\lambda
}}}|v|^{p}\varphi _{1}\varphi _{2}dx\,dt\Big)^{\frac{1}{pq}}, \\
&&\displaystyle\int_{Q_{T^{\lambda }}}|u|^{q}\varphi _{1}(t)\varphi
_{2}(x)\,dx\,dt+CT^{\lambda (1-\gamma _{2})}\displaystyle\int_{\mathbb{R}%
^{N}}v_{0}\varphi _{2}(x)dx \\
&\leq &\mathcal{B}\mathcal{A}^{1/p}\Big(\displaystyle\int_{Q_{T^{\lambda
}}}|u|^{q}\varphi _{1}\varphi _{2}dx\,dt\Big)^{\frac{1}{pq}}.
\end{eqnarray*}

Now, applying Young's inequality, we obtain 
\begin{eqnarray*}
&&(pq-1)\int_{0}^{T^{\lambda }}\int_{\mathbb{R}^{N}}|v|^{p}\varphi
_{2}(x)\varphi _{1}(t)\,dx\,dt+CpqT^{\lambda (1-\gamma _{1})}\int_{\mathbb{R}%
^{N}}u_{0}(x)\varphi _{2}(x)dx \\
&\leq &(pq-1)\big(\mathcal{A}\mathcal{B}^{1/q}\big)^{\frac{pq}{pq-1}}, \\
&&(pq-1)\displaystyle\int_{0}^{T^{\lambda }}\int_{\mathbb{R}%
^{N}}|u|^{q}\varphi _{2}(x)\varphi _{1}(t)\,dx\,dt+CpqT^{\lambda (1-\gamma
_{2})}\displaystyle\int_{\mathbb{R}^{N}}v_{0}(x)\varphi _{2}(x)dx \\
&\leq &(pq-1)\big(\mathcal{B}\,\mathcal{A}^{1/p}\big)^{\frac{pq}{pq-1}}.
\end{eqnarray*}%
At this stage, using the change of variables, $x=T^{2}y$, $t=T^{\lambda
}\tau $, with $\lambda >0$ to be chosen later, we obtain 
\begin{eqnarray*}
&&\displaystyle\int_{0}^{T^{\lambda }}\int_{\mathbb{R}^{N}}|v|^{p}\varphi
_{2}(x)\varphi _{1}(t)dx\,dt+T^{\lambda (1-\gamma _{1})}\displaystyle\int_{%
\mathbb{R}^{N}}u_{0}\varphi _{2}(x)dx \\
&\leq &C\Big(T^{-\lambda \gamma _{1}+(\lambda +2N)\frac{1}{q^{\prime }}%
}+T^{-4+(\lambda +2N)\frac{1}{q^{\prime }}}\Big)\Big(\displaystyle%
\int_{Q_{T^{\lambda }}}|u|^{q}\varphi _{1}(t)\varphi _{2}(x)dx\,dt\Big)^{1/q}
\\
&\leq &CT^{-\lambda \gamma _{1}+(\lambda +2N)\frac{1}{q^{\prime }}}\Big(%
\displaystyle\int_{Q_{T^{\lambda }}}|u|^{q}\varphi _{1}\varphi _{2}dx\,dt%
\Big)^{1/q}.
\end{eqnarray*}%
Analogously, we have 
\begin{eqnarray*}
&&\displaystyle\int_{0}^{T^{\lambda }}\int_{\mathbb{R}^{N}}|u|^{q}\varphi
_{2}(x)\varphi _{1}(t)dx\,dt+CT^{\lambda (1-\gamma _{2})}\displaystyle\int_{%
\mathbb{R}^{N}}v_{0}\varphi _{2}(x)dx \\
&\leq &C\Big(T^{-\lambda \gamma _{2}+(\lambda +2N)\frac{1}{p^{\prime }}%
}+T^{-4+(\lambda +2N)\frac{1}{p^{\prime }}}\Big)\Big(\displaystyle%
\int_{Q_{T^{\lambda }}}|v|^{p}\varphi _{1}\varphi _{2}dx\,dt\Big)^{1/p} \\
&=&CT^{-\lambda \gamma _{2}+(\lambda +2N)\frac{1}{p^{\prime }}}\Big(%
\displaystyle\int_{Q_{T^{\lambda }}}|v|^{p}\varphi _{1}\varphi _{2}dx\,dt%
\Big)^{1/p}.
\end{eqnarray*}%
Choosing $\gamma _{1}\lambda =4$, we have 
\begin{eqnarray*}
&&\displaystyle\int_{0}^{T^{\lambda }}\int_{\mathbb{R}^{N}}|v|^{p}\varphi
_{2}(x)\varphi _{1}(t)dx\,dt+T^{\lambda (1-\gamma _{1})}\displaystyle\int_{%
\mathbb{R}^{N}}u_{0}\varphi _{2}(x)dx \\
&\leq &C\Big(T^{-\frac{4}{q\gamma _{1}}\gamma _{2}+(\lambda +2N)\frac{1}{%
p^{\prime }}\frac{1}{q}-4+(\frac{4}{\gamma _{1}}+2N)\frac{1}{q^{\prime }}%
}+T^{-4\frac{1}{q}+(\lambda +2N)\frac{1}{p^{\prime }q}-4+(\frac{4}{\gamma
_{1}}+2N)\frac{1}{q^{\prime }}}\Big) \\
&&\quad \times \Big(\int_{Q_{T^{\lambda }}}|v|^{p}\varphi _{1}\varphi
_{2}dx\,dt\Big)^{\frac{1}{pq}},
\end{eqnarray*}%
and 
\begin{eqnarray*}
&&\int_{0}^{T^{\lambda }}\int_{\mathbb{R}^{N}}|u|^{q}\varphi _{2}(x)\varphi
_{1}(t)dx\,dt+CT^{\lambda (1-\gamma _{2})}\int_{\mathbb{R}^{N}}v_{0}\varphi
_{2}(x)dx \\
&\leq &C\Big(T^{-\lambda \gamma _{2}+(\lambda +2N)\frac{1}{p^{\prime }}%
}+T^{-4+(\lambda +2N)\frac{1}{p^{\prime }}}\Big)\Big(\displaystyle%
\int_{Q_{T^{\lambda }}}|v|^{p}\varphi _{1}\varphi _{2}dx\,dt\Big)^{1/p} \\
&\leq &C\Big(T^{-\lambda \gamma _{2}+(\lambda +2N)\frac{1}{p^{\prime }}%
}+T^{-4+(\lambda +2N)\frac{1}{p^{\prime }}}\Big)T^{-4\frac{1}{p}+(\frac{4}{%
\gamma _{1}}+2N)\frac{1}{pq^{\prime }}} \\
&&\quad \times \left( \displaystyle\int_{Q_{T^{\lambda }}}|u|^{q}\varphi
_{1}\varphi _{2}dx\,dt\right) ^{\frac{1}{pq}} \\
&=&C\Big(T^{-\frac{4}{\gamma _{1}}\gamma _{2}+(\frac{4}{\gamma _{1}}+2N)%
\frac{1}{p^{\prime }}-4\frac{1}{p}+(\frac{4}{\gamma _{1}}+2N)\frac{1}{%
pq^{\prime }}}+T^{-4+(\frac{4}{\gamma _{1}}+2N)\frac{1}{p^{\prime }}-4\frac{1%
}{p}+(\frac{4}{\gamma _{1}}+2N)\frac{1}{pq^{\prime }}}\Big) \\
&&\quad \times \Big(\displaystyle\int_{Q_{T^{\lambda }}}|u|^{q}\varphi
_{1}(t)\varphi _{2}dx\,dt\Big)^{\frac{1}{pq}}.
\end{eqnarray*}%
Therefore, using the $\varepsilon $-Young inequality, we obtain 
\begin{eqnarray}
\int_{\mathbb{R}^{N}}u_{0}(x)\varphi _{2}(x)dx &\leq &CT^{\delta _{1}},
\label{cigm1} \\
\int_{\mathbb{R}^{N}}v_{0}(x)\varphi _{2}(x)dx &\leq &CT^{\delta _{2}},
\label{cigm2}
\end{eqnarray}%
where 
\begin{eqnarray*}
\delta _{1} &=&\max \Big\{(-\frac{4}{q\gamma _{1}}\gamma _{2}+(\frac{4}{%
\gamma _{1}}+2N)\frac{1}{p^{\prime }}\frac{1}{q}-4+(\frac{4}{\gamma _{1}}+2N)%
\frac{1}{q^{\prime }})\frac{pq}{pq-1}+\frac{4}{\gamma _{1}}(\gamma _{1}-1),
\\
&&\quad (-4\frac{1}{q}+(\frac{4}{\gamma _{1}}+2N)\frac{1}{p^{\prime }q}-4+(%
\frac{4}{\gamma _{1}}+2N)\frac{1}{q^{\prime }})\frac{pq}{pq-1}+\frac{4}{%
\gamma _{1}}(\gamma _{1}-1)\Big\},
\end{eqnarray*}%
and 
\begin{eqnarray*}
\delta _{2} &=&\max \Big\{(-\frac{4}{\gamma _{1}}\gamma _{2}+(\frac{4}{%
\gamma _{1}}+2N)\frac{1}{p^{\prime }}-4\frac{1}{p}+(\frac{4}{\gamma _{1}}+2N)%
\frac{1}{pq^{\prime }})\frac{pq}{pq-1}+\frac{4}{\gamma _{1}}(\gamma _{2}-1),
\\
&&\quad (-4+(\frac{4}{\gamma _{1}}+2N)\frac{1}{p^{\prime }}-4\frac{1}{p}+(%
\frac{4}{\gamma _{1}}+2N)\frac{1}{pq^{\prime }})\frac{pq}{pq-1}+\frac{4}{%
\gamma _{1}}(\gamma _{2}-1)\Big\}.
\end{eqnarray*}

The condition \eqref{critdimension} leads to either $\delta _{1}<0$ or $%
\delta _{2}<0$. Then as $T\rightarrow \infty $, the right-hand side of %
\eqref{cigm1} (resp. \eqref{cigm2}) tends to zero while the left-hand side
tends to $\int_{\mathbb{R}^{N}}u_{0}(x)dx>0$ (resp. $\int_{\mathbb{R}%
^{N}}v_{0}(x)dx>0)$; a contradiction.

We repeat the same argument for $\gamma _2 \lambda =4$ to conclude the proof
of Theorem \ref{NEG}. \smallskip

\noindent \textbf{Case $p=1,q>1$} (the case $p>1,q=1$ is treated similarly).
We still use the weak formulation of the solution and argue by
contradiction. Let us set 
\begin{equation*}
\mathcal{I}=\int_{0}^{T}\int_{\Omega }v\varphi \,dx\,dt,\quad \mathcal{J}%
=\left( \int_{0}^{T}\int_{\Omega }u^{q}\varphi \,dx\,dt\right) ^{1/q}.
\end{equation*}%
Then, applying Holder's inequality as above, we obtain 
\begin{equation}
\mathcal{I}+\int_{0}^{T}\int_{\Omega }u_{0}D_{t|T}^{\gamma _{1}}\varphi
\,dx\,dt\leq \mathcal{J}(\mathcal{A}+\mathcal{B}),  \label{Epsil}
\end{equation}%
where 
\begin{equation*}
\mathcal{A}=\Big(\displaystyle\int_{0}^{T}\displaystyle\int_{\Omega }\varphi
^{-\frac{q^{\prime }}{q}}|\Delta \varphi |^{q^{\prime }}\,dx\,dt\Big)%
^{1/q^{\prime }},\quad \mathcal{B}=\Big(\int_{0}^{T}\int_{\Omega }\varphi ^{-%
\frac{q^{\prime }}{q}}|D_{t|T}^{\gamma _{1}}\varphi |^{q^{\prime }}\,dx\,dt%
\Big)^{1/q^{\prime }}\,;
\end{equation*}%
and 
\begin{eqnarray}
\mathcal{J}^{q}+\int_{0}^{T}\int_{\Omega }v_{0}\,D_{t|T}^{\gamma
_{1}}\varphi \,dx\,dt &\leq &\lambda \int_{0}^{T}\int_{\Omega }v\varphi
\,dx\,dt+\int_{0}^{T}\int_{\Omega }vD_{t|T}^{\gamma _{2}}\varphi \,dx\,dt 
\notag \\
&\leq &(\lambda +\varepsilon )\mathcal{I},  \label{lambd}
\end{eqnarray}%
thanks to the $\varepsilon $-Young inequality and where we have chosen $%
\varphi $ as a the first eigen-function of the spectral problem 
\begin{equation*}
-\Delta \varphi =\lambda \varphi ,x\in B_{T}(0),\quad \varphi _{_{|\partial
\Omega }}=0\,,
\end{equation*}%
where ($\Omega =B_{T}(0)\subset \mathbb{R}^{N}$ is the ball centered in zero
and of radius $T$ and $\partial \Omega $ is the boundary of $\Omega $).
Adding equation \eqref{lambd} to ($\lambda +\varepsilon $) times equation %
\eqref{Epsil}, we obtain 
\begin{equation*}
\mathcal{J}^{q}+(\lambda +\varepsilon )\int_{0}^{T}\int_{\Omega
}u_{0}D_{t|T}^{\gamma _{1}}\varphi \,dx\,dt+\int_{0}^{T}\int_{\Omega
}v_{0}\,D_{t|T}^{\gamma _{2}}\varphi \,dx\,dt\leq (\lambda +\varepsilon )%
\mathcal{J}(\mathcal{A}+\mathcal{B}),
\end{equation*}%
whereupon, 
\begin{equation*}
\mathcal{J}^{q-1}\leq \mathcal{A}+\mathcal{B}.
\end{equation*}%
Replacing $\varphi (x)$ by $\varphi (\frac{x}{T})$ and passing to the new
variables $y=T^{-1}x$ and $\tau =T^{-1}t$, and then letting $T$ go to
infinity, we obtain a contradiction whenever $q<1+\frac{2}{N}$.
\end{proof}

\begin{proof}[Proof of Theorem \protect\ref{Blwr}]
Let $u_{0},v_{0}\in C_{0}(\mathbb{R}^{N})$ be nonnegative and $\left(
u,v\right) $ be the corresponding solution of (\ref{sys1})-(\ref{initdat}).
We proceed by contradiction. Assume that $\left( u,v\right) $ exists
globally in time, that is $\left( u,v\right) $ exists in $\left( 0,t_{\ast
}(u_{0},v_{0})\right) $, for any $t_{\ast }(u_{0},v_{0})>0$. Let $T\in
(0,t_{\ast }(u_{0},v_{0}))$ be arbitrarily fixed.

Taking $\chi $ as test-function and setting 
\begin{equation*}
X(t):=\int_{\mathbb{R}^{N}}u(t,x)\chi (x)dx,\quad Y(t):=\int_{\mathbb{R}%
^{N}}v(t,x)\chi (x)dx\quad Z(t)=\int_{\Omega }\chi \left(
u(t,x)+v(t,x)\right) dx,
\end{equation*}%
and 
\begin{equation*}
Z_{0}=\int_{\mathbb{R}^{N}}\left( u_{0}+v_{0}\right) \chi \left( x\right)
\,dx.
\end{equation*}%
It follows from (\ref{sys1})-(\ref{initdat}) that%
\begin{equation}
\left\{ 
\begin{array}{c}
{}^{C}D_{0|t}^{\gamma }\displaystyle\int_{\mathbb{R}^{N}}u(t,x)\chi (x)dx-%
\displaystyle\int_{\mathbb{R}^{N}}u\Delta \chi (x)dx=\displaystyle\int_{%
\mathbb{R}^{N}}\left\vert v(t,x)\right\vert ^{p}\chi (x)dx,\;t\in (0,T), \\ 
^{C}D_{0|t}^{\gamma }\displaystyle\int_{\mathbb{R}^{N}}v(t,x)\chi (x)dx-%
\displaystyle\int_{\mathbb{R}^{N}}v\Delta \chi (x)dx=\displaystyle\int_{%
\mathbb{R}^{N}}\left\vert u(t,x)\right\vert ^{q}\chi (x)dx,\text{ }t\in
(0,T),%
\end{array}%
\right.  \label{systtesto}
\end{equation}%
supplemented with the initial conditions 
\begin{equation}
X(0)=\int_{\mathbb{R}^{N}}u_{0}(x)\chi (x)dx,\quad Y(0)=\int_{\mathbb{R}%
^{N}}v_{0}(x)\chi (x)dx.  \label{systtestotwb}
\end{equation}%
From (\ref{systtesto})-(\ref{systtestotwb}), we have 
\begin{equation}
D_{0|t}^{\gamma }\left( [Z-Z_{0}]\right) (t)-\int_{\mathbb{R}^{N}}\left(
u(t,x)+v(t,x)\right) \Delta \chi (x)dx=\int_{\mathbb{R}^{N}}\left(
\left\vert v(t,x)\right\vert ^{p}+\left\vert u(t,x)\right\vert ^{q}\right)
\chi (x)dx,\;t\in (0,T).  \label{ineqfracsy}
\end{equation}%
We observe that 
\begin{equation*}
\int_{\mathbb{R}^{N}}v(x,t)\chi (x)\,dx=\int_{\mathbb{R}^{N}}v(x,t)\chi ^{%
\frac{1}{p}}(x)\chi ^{1-\frac{1}{p}}(x)\,dx.
\end{equation*}%
Since the function $\chi $ satisfies $\displaystyle\int_{\mathbb{R}^{N}}\chi
(x)\,dx=1$, then it yields by H\"{o}lder's inequality that 
\begin{equation*}
\displaystyle\int_{\mathbb{R}^{N}}v(x,t)\chi (x)\,dx\leq \left( \displaystyle%
\int_{\mathbb{R}^{N}}\left\vert v(x,t)\right\vert ^{p}\chi (x)\,dx\right) ^{%
\frac{1}{p}}.
\end{equation*}%
So 
\begin{equation}
\int_{\mathbb{R}^{N}}\left\vert v(x,t)\right\vert ^{p}\chi (x)\,dx\geq
\left( \int_{\mathbb{R}^{N}}v(x,t)\chi (x)\,dx\right) ^{p}=Y^{p}(t).
\label{addhin}
\end{equation}%
Similarly, we obtain%
\begin{equation}
\int_{\mathbb{R}^{N}}\left\vert u(x,t)\right\vert ^{q}\chi (x)\,dx\geq
\left( \int_{\mathbb{R}^{N}}u(x,t)\chi (x)\,dx\right) ^{q}=X^{q}(t).
\label{addhinb}
\end{equation}%
Using estimates \eqref{addhin}, \eqref{addhinb} in \eqref{ineqfracsy} and
the fact that the funtion $\chi $ satisfies $\Delta \chi \geq -\chi $, it
yields%
\begin{equation}
D_{0|t}^{\gamma }\left( [Z-Z_{0}]\right) +Z(t)\geq Y^{p}(t)+X^{q}(t)\;t\in
(0,T).  \label{frdin}
\end{equation}%
By adding $Z(t)$ to the two members of \eqref{frdin}, we get 
\begin{equation*}
D_{0|t}^{\gamma }\left( [Z(t)-Z_{0}]\right) +2Z(t)\geq
Y^{p}(t)+X^{q}(t)+X(t)+Y(t)\geq Y^{p}(t)+X^{q}(t)+X(t).
\end{equation*}%
We assume that $q\geq p$, by using the fact that $X^{q}(t)+X(t)\geq X^{p}(t)$
and 
\begin{equation*}
(a+b)^{r}\leq 2^{r-1}(a^{r}+b^{r}),\quad a,b>0,\;r\geq 1,
\end{equation*}%
we get 
\begin{equation}
D_{0|t}^{\gamma }\left( [Z(t)-Z_{0}]\right) +2Z(t)\geq 2^{1-p}Z(t)^{p}\text{.%
}  \label{OFI}
\end{equation}%
We put $F(y)=2^{1-p}y^{p}-2y$, the function $F$ is convex on $\left(
0,\infty \right) $ (since $F\in C^{2}\left( 0,+\infty \right) ,$ $F^{\prime
\prime }\geq 0)$.

Writing $\partial _{t}\left( k\ast \lbrack Z-Z_{0}]\right) (t)$ instead of $%
D_{0|t}^{\gamma }\left( [Z(t)-Z_{0}]\right) $ with $k(t)=\frac{t^{-\gamma }}{%
\Gamma \left( 1-\gamma \right) }$ in (\ref{OFI}), we get 
\begin{equation}
\partial _{t}\left( k\ast \lbrack Z-Z_{0}]\right) (t)\geq F(Z(t)),\quad 
\text{ }t\in (0,T).  \label{WZCA}
\end{equation}%
It is clear that $F(y)>0$ and $F^{\prime }(y)>0$ for all $y>2^{\frac{p}{p-1}%
}:=\alpha _{1}$.

Suppose now that $Z_{0}>\alpha _{1}$. We claim that (\ref{WZCA}) implies
that $Z(t)>\alpha _{1}$ for all\ $t\in (0,T)$.

In fact, for $Z(0)=Z_{0}>\alpha _{1}$, we have by continuity of $Z$, there
exists $\delta \in (0,T]$ such that $Z(t)>\alpha _{1}$ for all $t\in \left(
0,\delta \right) $. This implies that $F(Z(t))>0$ for all\ $t\in (0,\delta
). $

By the comparison principle, it yields that $Z(t)\geq Z_{0}$ for any\ $t\in
(0,\delta )$.

Setting 
\begin{equation*}
\delta _{1}:=\sup \left\{ s\in (0,T):\,Z(t)\geq Z_{0}\,\;\;\;t\in
(0,s)\right\} ,
\end{equation*}%
then $\delta _{1}>0$. We want to show that $\delta _{1}=T$.

Indeed, if $\delta _{1}<T$, then by setting $s=t-\delta _{1}$ for $t\in
(\delta _{1},T)$ and $\tilde{Z}(s)=Z(s+\delta _{1})$, $s\in (0,T-\delta
_{1}) $, it follows from positivity of $Z-Z_{0}$ on $(0,\delta _{1})$ and
nonincreasness of $k$ that

\begin{equation}
\partial _{s}\left( k\ast \lbrack \tilde{Z}-Z_{0}]\right) (s)\geq \partial
_{t}\left( k\ast \lbrack Z-Z_{0}]\right) (s+\delta _{1}),\quad \;s\in
(0,T-\delta _{1}).  \label{shift}
\end{equation}%
From (\ref{WZCA}) and (\ref{shift}) we deduce that 
\begin{equation*}
\partial _{s}\left( k\ast \lbrack \tilde{Z}-Z_{0}]\right) (s)\geq F(\tilde{Z}%
(s)),\quad \;s\in (0,T-\delta _{1}).
\end{equation*}%
This time-shifting property can be already found in \cite{Zacher}. So we may
repeat the argument from above to see that there exists $\tilde{\delta}\in
(0,T-\delta _{1}]$ such that $\tilde{Z}(s)\geq Z_{0}$ for any\ $s\in (0,%
\tilde{\delta})$. This leads to a contradiction with the definition of $%
\delta _{1}$.

Hence, the assumption $\delta _{1}<T$ was not true. This proves the claim.
Knowing that $Z(t)\geq Z_{0}>\alpha _{1}$ for any\ $t\in (0,T)$ it follows
from (\ref{WZCA}) that 
\begin{equation}
^{C}D_{0|t}^{\gamma }Z(t)=\partial _{t}\left( k\ast \lbrack Z-Z_{0}]\right)
(t)\geq F(Z(t))>0,\;\text{for any }t\in (0,T).  \label{Wsuper}
\end{equation}%
Therefore the function $Z(t)$ satisfying (\ref{Wsuper}) is upper solution of
the following problem%
\begin{equation}
^{C}D_{0|t}^{\gamma }y=F(y)=2^{1-p}y^{p}-2y,\text{ }y\left( 0\right) =Z_{0},
\label{FODE}
\end{equation}

we have by comparison principle $Z\left( t\right) \geq y(t)$ (see \cite[%
Theorem 2.3]{Lakshmikantham},\cite[Theorem 4.10.]{Li}).

On the other hand, since $F\left( 0\right) \geq 0,$ $F(y)>0$ and $F^{\prime
}(y)>0$, for all $y\geq Z_{0}>2^{\frac{p}{p-1}}$. It then follows from Lemma %
\ref{vatsala lem}, that $v(t)=w\left( \frac{t^{\gamma }}{\Gamma \left(
\gamma +1\right) }\right) $ is a lower solution for (\ref{FODE})\ (which
means%
\begin{equation*}
^{C}D_{0|t}^{\gamma }v\leq F(v)=2^{1-p}v^{p}-2v,\text{ }v\left( 0\right)
=Z_{0}\leq Z_{0}),
\end{equation*}%
where $w\left( t\right) $ solves the following ODE%
\begin{equation}
\frac{dw}{dt}=F(w)=2^{1-p}w^{p}-2w,\text{ }w\left( 0\right) =Z_{0}.
\label{Vatsa}
\end{equation}%
By comparison principle (see \cite[Theorem 2.3]{Lakshmikantham},\cite[%
Theorem 4.10.]{Li}), we get $y(t)\geq v(t).$

So, by solving the Cauchy problem (\ref{Vatsa}), which is equivalent to\ 
\begin{equation*}
\frac{d}{dt}\left( e^{2t}w\right) =2^{1-p}e^{2(1-p)t}\left( e^{2t}w\right)
^{p},\text{ }w(0)=Z(0),
\end{equation*}%
in which the explicit blow-up solution given by%
\begin{equation*}
w\left( t\right) =\left( \frac{e^{2(1-p)t}-1}{2^{p}}+Z_{0}^{1-p}\right) ^{%
\frac{1}{1-p}}e^{-2t}
\end{equation*}%
whose blows up in finite time $t_{\ast \ast }=\frac{\ln \left(
1-2^{p}Z_{0}^{1-p}\right) }{2(1-p)}$. By comparison principle (see \cite[%
Theorem 2.3]{Lakshmikantham},\cite[Theorem 4.10.]{Li}), we conclude that 
\begin{equation*}
Z(t)\geq y(t)\geq v(t)=w\left( \frac{t^{\gamma }}{\Gamma \left( \gamma
+1\right) }\right) =\left( \frac{e^{2(1-p)\frac{t^{\gamma }}{\Gamma \left(
\gamma +1\right) }}-1}{2^{p}}+Z_{0}^{1-p}\right) ^{\frac{1}{1-p}}e^{-2\frac{%
t^{\gamma }}{\Gamma \left( \gamma +1\right) }}
\end{equation*}%
which in turn leads to $Z(t)$ blows-up in finite time at $\bar{t}_{\ast \ast
}\leq \left[ \frac{\ln \left( 1-2^{p}Z_{0}^{1-p}\right) }{2(1-p)}\Gamma
\left( \gamma +1\right) \right] ^{1/\gamma }$. Thus the same holds for the
solution $\left( u,v\right) $ of (\ref{sys1})-(\ref{initdat}), which in turn
leads to a contradiction.
\end{proof}

\begin{remark}
\textrm{Similar results were obtained in \cite[\textrm{Theorem 3.5}]%
{ZhangQuanLi} using another method, while the authors did not address the
estimation of the time blowing up. }
\end{remark}

\subsection*{Acknowledgements}

This project was funded by the Deanship of Scientific Research (DSR) at King
Abdulaziz University, Jeddah, Saudi Arabia, under grant no. (RG-82-130-38).
The authors, therefore, acknowledge with thanks DSR technical and financial
support.

\begin{flushleft}
\textbf{Ahmad Bashir, Ahmed Alsaedi}\newline
NAAM Research Group, Department of Mathematics,\newline
Faculty of Science, King Abdulaziz University, P.O. Box 80203 Jeddah 21589,%
\newline
Saudi Arabia.\newline

\textbf{Mohamed Berbiche}\newline
Laboratory of Mathematical Analysis, Probability and Optimizations,\newline
Mohamed Khider University, Biskra,Po. Box 145 Biskra ( 07000),\newline
Algeria.\newline

\textbf{Mokhtar Kirane}\newline
NAAM Research Group, Department of Mathematics,\newline
Faculty of Science, King Abdulaziz University, P.O. Box 80203 Jeddah 21589,%
\newline
Saudi Arabia. LaSIE, Universit\'{e} de La Rochelle, P\^{o}le Sciences et
Technologies,\newline
Avenue Michel Cr\'{e}peau, 17000 La Rochelle,\newline
France.
\end{flushleft}

\end{document}